\documentclass[12pt,draftcls,onecolumn]{IEEEtran}

\usepackage{amssymb,amsmath,amsfonts,amsthm}
\usepackage{algorithm,algorithmic}
\usepackage{boxedminipage}
\usepackage{cite}
\usepackage[dvips]{graphicx}
\usepackage{epsfig}
\usepackage{float}
\graphicspath{{figures/}}
\usepackage{graphicx}
\usepackage{graphics}
\usepackage{multirow}
\usepackage{threeparttable} % package for footnotes in tables
\usepackage[usenames,dvipsnames]{color}
\usepackage{url,bm,xspace,dsfont}

\usepackage{verbatim}

\newtheorem{theorem}{Theorem}
\newtheorem{lemma}{Lemma}
\newtheorem{definition}{Definition}
\newtheorem{corollary}{Corollary}

\newtheorem{remark}{Remark}

\newtheorem{assumptions}{Assumptions}
\newtheorem{problem}{Problem}

\newcommand{\veps}{\varepsilon}
\newcommand{\la}{\langle}
\newcommand{\ra}{\rangle}

\newcommand{\sr}{\stackrel}

\newcommand{\rar}{\rightarrow}

\newcommand{\tri}{\sr{\triangle}{=}}

\newcommand{\be}{\begin{equation}}
\newcommand{\ee}{\end{equation}}
\newcommand{\bea}{\begin{eqnarray}}
\newcommand{\eea}{\end{eqnarray}}
\newcommand{\bes}{\begin{eqnarray*}}
\newcommand{\ees}{\end{eqnarray*}}

\newcommand{\bi}{\begin{itemize}}
\newcommand{\ei}{\end{itemize}}
\newcommand{\ben}{\begin{enumerate}}
\newcommand{\een}{\end{enumerate}}

\newcommand{\bp}{\begin{problem}}
\newcommand{\ep}{\end{problem}}
\newcommand{\hso}{\hspace{.1in}}
\newcommand{\hst}{\hspace{.2in}}

\newcommand{\noi}{\noindent}

\newcommand{\bc}{\begin{center}}
\newcommand{\ec}{\end{center}}

\floatstyle{ruled}
\newfloat{algorithm}{H}{algo}
\floatname{algorithm}{\footnotesize Algorithm}

\begin{document}

\title{\bf Centralized Versus Decentralized Team Games of  Distributed Stochastic Differential Decision Systems with Noiseless Information Structures-Part I: General Theory}

\begin{comment}

\title{\bf Team Games of  Distributed Stochastic Differential Decentralized Decision Systems with Noiseless Information Structures-Part I: General Theory}

\end{comment}


%\vspace*{1.0cm}

\author{ Charalambos D. Charalambous\thanks{C.D. Charalambous is with the Department of Electrical and Computer Engineering, University of Cyprus, Nicosia 1678 (E-mail: chadcha@ucy.ac.cy).}  \: and Nasir U. Ahmed\thanks{N.U Ahmed  is with the School of  Engineering and Computer Science, and Department of Mathematics,
University of Ottawa, Ontario, Canada, K1N 6N5 (E-mail:  ahmed@site.uottawa.ca).}
}

\maketitle

\begin{abstract}
Decentralized optimization of distributed stochastic differential systems has been an active area of research for over half a century. Its formulation utilizing static team and person-by-person optimality criteria is  well investigated. However, the results have not been generalized to nonlinear distributed stochastic differential systems  possibly due  to technical difficulties inherent with   decentralized decision strategies.

In this first part of the two-part paper,  we derive team optimality and person-by-person optimality conditions for distributed stochastic differential systems with different  information structures. The optimality conditions are given in terms of a Hamiltonian system of equations described by a system of coupled backward and forward stochastic differential equations and  a conditional Hamiltonian, under both regular  and relaxed strategies. Our methodology is based on the semi martingale representation theorem and variational methods.   Throughout the presentation  we discuss similarities to optimality conditions of centralized decision making.

\end{abstract}

%\vspace*{1.5cm}

\begin{comment}
  \vskip6pt\noindent {\bf Key Words.}  Decentralized, team games, stochastic differential systems, stochastic maximum principle.

  \vspace*{1.5cm}

   \vskip6pt\noindent{\bf  2000 AMS Subject Classification} 49J55, 49K45, 93E20.
\end{comment}

%\begin{IEEEkeywords}
\noi{\bf Index Terms.} Team and Person-by-Person Optimality, Stochastic Differential Systems, Stochastic Maximum Principle, Relaxed Strategies.

%\end{IEEEkeywords}

%\newpage

  \section{Introduction}
\label{introduction}
Over the last 50 years many mathematical concepts and procedures were developed to design  optimal  control strategies for stochastic dynamical  systems. We refer to this set of mathematical concepts and procedures as  the "classical theory of stochastic optimization". It has been  utilized extensively to address the questions of   existence of optimal strategies,  and necessary and sufficient optimality conditions for   systems driven by continuous martingale processes (Brownian motion processes), and discontinuous martingale processes (jump processes). It has been  successfully applied to centralized fully observable control problems, meaning   the admissible strategies are functions of a common noiseless measurements of the system  \cite{fleming-rischel1975,elliott1977,bismut1978,ahmed1981,elliott1982,elliott-kohlmann1994,peng1990,yong-zhou1999,ahmed2006},  and to centralized partially observable control systems, meaning the admissible strategies are functions of common noisy measurements of the system  \cite{elliott1977,bensoussan1983,bensoussan1992a,charalambous-hibey1996,ahmed-charalambous2007}. In addition, optimility conditions are derived for infinite dimensional systems and impulsive systems in \cite{ahmed1981,ahmed2005,ahmed2006}. Thus, the classical theory of optimization is developed  on the assumption of centralized decisions or control actions. It  presupposes that all information about the system can be acquired  and accordingly the decision policies (control actions) can be formulated. The basic underlying assumption is that the acquisition of the information is centralized or the information acquired at different  locations  is communicated to each decision maker or control.  \\

When the system model consists of multiple decision makers,  and  the acquisition of information and its processing is decentralized or shared among several locations,  the   decision makers  actions are based on different information.  We call  the information available for such  decisions,  "decentralized information structures or patterns".  When the system model is dynamic, consisting of an interconnection of at least two subsystems, and  the decisions are based on decentralized information structures, we call the overall system  a "distributed system with decentralized information structures". Over the years several specific forms of decentralized information structures   are analyzed mostly in discrete-time
\cite{witsenhausen1968,witsenhausen1971,ho-chu1972,kurtaran-sivan1973,sandell-athans1974,kurtaran1975,varaiya-walrand1978,ho1980,krainak-speyer-marcus1982a,krainak-speyer-marcus1982b,bansal-basar1987,waal-vanschuppen2000}, and more recently \cite{bamieh-voulgaris2005,aicardi-davoli-minciardi1987,nayyar-mahajan-teneketzis2011,vanschuppen2011,lessard-lall2011,mahajan-martins-rotkowitz-yuksel2012}.    However,   at this stage there is no systematic framework  addressing  optimality conditions for distributed systems with decentralized information structures. The absence of such optimization theory raises the question whether  the classical theory of optimization  is limited in mathematical concepts and procedures to deal with decentralized  systems.  \\

In this first part of the two-part investigation, we show that the classical theory of optimization does not have such a limitation.  We consider a team game reward  \cite{marschak1955,radner1962,marschak-radner1972,krainak-speyer-marcus1982a,waal-vanschuppen2000}  and we apply concepts from  the classical theory of optimization to derive    necessary and sufficient optimality conditions for   nonlinear stochastic distributed systems with decentralized information structures. Our methodology utilizes the semi martingale representation  theorem and variational methods recently reported  by the authors in \cite{ahmed-charalambous2012a}.   \\

The optimality conditions developed in this paper can be applied to  many  architectures of distributed systems such as Fig.~\ref{ADC} (see also  \cite{vanschuppen2012}). Each decision maker makes its decision based on local information and exerts control action that affects  the overall distributed system, without allowing  communication between the  local decision makers. Such systems are called  distributed systems with decentralized information structures.  The team formulation of the distributed  system with decentralized information structures,  consists of  an interconnection of $N$ subsystems.
Each subsystem $i$ has its state denoted by  $x^i \in {\mathbb X}^i$,  a local decision maker or control input $u^i \in {\mathbb A}^i$, an exogenous Brownian motion noise input  $W^i \in {\mathbb W}^i$, and a coupling from the other subsystem. \\

\begin{figure}
\begin{center}
\includegraphics[bb=-120 20 700 150,scale=0.5]{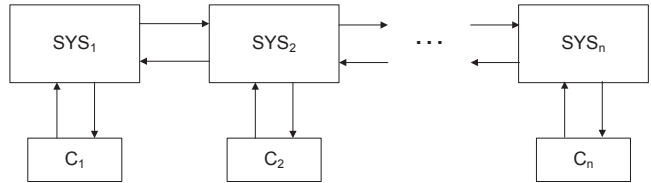}
\caption{ Diagram of architecture for distributed stochastic differential decision systems.}
\label{ADC}
\end{center}
\end{figure}

\noi {\bf Decentralized  Information Structures for  Decision Makers}\\  

 The information structures of the local decision makers  $u^i, i =1,2,\ldots, N$  are defined as follows. For any $t \in [0,T]$, the information structure available to  decision maker (DM) $u^i$ is modeled by  the $\sigma-$algebra ${\cal G}_{0,t}^i$ generated by the observable events associated with the local subsystem. These observables can be generated by nonanticipative functionals of the noise entering the system, nonanticipative functions of  the  state of the system, its  delayed versions, or  any possible  combinations thereof. Let us denote the admissible strategies of $u^i$ with action spaces ${\mathbb A}^i$, by ${\mathbb U}^i[0,T], i=1,2,\ldots, N$ (meaning that  $u^i$ is a nonanticipative measurable functional of the information algebra  ${\cal G}_T^i\tri \{{\cal G}_{0,t}^i: t \in [0,T]\}$  taking values from  ${\mathbb A}^i$. Thus the augmented state, control and noise  of the decentralized system can be written as  $$x\tri (x^1,x^2,\ldots,x^N) \in {\mathbb X}^{(N)}, \hst  u\tri (u^1,u^2, \ldots, u^N) \in {\mathbb A}^{(N)}, \hst W\tri (W^1,W^2,\ldots,W^N) \in {\mathbb W}^{(N)}.$$
Then the overall  system can be  expressed in compact form by the following stochastic It\^o differential equation
\bea
dx(t) = f(t,x(t),u_t)dt + \sigma(t,x_t,u_t)dW(t), \hst x(0)=x_0, \hst t \in (0,T]. \label{ds2}
 \eea
 
\noi {\bf Team Game Pay-off  Functional } \\

The objective is to find a team optimal strategy  $ u^o \equiv (u^{1,o}, \ldots, u^{N, o}) \in \times_{i=1}^N {\mathbb U}^i[0,T]$ at  which  the pay-off functional defined by
\begin{align} J(u^o) \equiv
J(u^{1,o}, \ldots, u^{N,o}) \tri
\inf_{ (u^1\ldots u^N) \in \times_{i=1}^N {\mathbb U}^i[0,T] } {\mathbb E} \Big\{ \int_{0}^T \ell(t,x(t),u(t))dt  +  \varphi(x(T))     \Big\}   \label{ds4}
\end{align} attains its minimum.
\\

  We consider two main classes  of decentralized noiseless information structures;  1) nonanticipative functionals of any subset  of the  sybsystems Brownian motions $\{W^1, \ldots, W^N\}$, called "nonanticipative information structures", and 2) nonanticipative functionals of any subset  of the subsystem states $\{x^1, \ldots, x^N\}$, called "feedback information structures" (see Section~\ref{tpbp}).\\

\noi {\bf Team Game Optimality Conditions}\\

 In Section~\ref{regular} we derive team optimality conditions (Theorem~\ref{theorem7.1}) for pay-off (\ref{ds4}) subject to (\ref{ds2}), under a strong formulation of the  filtered  probability space $\Big(\Omega,{\mathbb F},\{ {\mathbb F}_{0,t}:   t \in [0, T]\}, {\mathbb P}\Big)$. These are summarized below.\\
Define the Hamiltonian
\bes
 {\cal  H}: [0, T] \times {\mathbb X}^{(N)}\times {\mathbb X}^{(N)}\times {\cal L}({\mathbb W}^{(N)}, {\mathbb X}^{N)})\times {\mathbb A}^{(N)} \longrightarrow {\mathbb R}
\ees
   by
   \begin{align}
    {\cal H} (t,\xi,\zeta,M,\nu) \tri    \langle f(t,\xi,\nu),\zeta \rangle + tr (M^*\sigma(t,\xi,\nu))
     + \ell(t,\xi,\nu),  \hst  t \in  [0, T]. \label{h1i}
    \end{align}
    \noi For any $u \in {\mathbb U}^{(N)} \equiv  \times_{i=1}^N {\mathbb U}^i[0,T]$, consider  the adjoint process $\{\psi, Q\}$  and the state $x$ satisfying the following backward  and forward stochastic differential equations respectively,
\begin{align}
d\psi (t)  &= - {\cal H}_x (t,x(t),\psi(t),Q(t),u_t) dt + Q(t)dW(t),  \hst \psi(T)=  \varphi_x(x(T)) , \hso t \in [0,T), \label{h2i}   \\
dx(t)         &= {\cal H}_\psi (t,x(t),\psi(t),Q(t),u_t)     dt + \sigma(t,x(t),u_t) dW(t), \hst        x(0) =  x_0, \hst t \in (0,T]. \label{h3i}
 \end{align}
The stochastic optimality conditions of the team game with decentralized noiseless information structures  are given  below.

 \begin{description}
\item[(1)] {\bf Necessary Conditions.}  Under certain conditions, which are precisely those of the classical theory of optimization,  the following hold.  \\ For  an element $u^o \in {\mathbb U}^{(N)} \equiv  \times_{i=1}^N {\mathbb U}^i[0,T]$ with the corresponding solution $x^o$  to be team optimal, it is necessary  that
the following hold:

\noi The process $\{\psi^o,Q^o\}$ is the  unique solution of the backward stochastic differential equation (\ref{h2i}) corresponding to the pair $\{u^o,x^o\}$ and that they together  satisfy   the  point wise almost sure inequalities with respect to the $\sigma$-algebras ${\cal G}_{0,t}^i$, $ t\in [0, T], i=1, 2, \ldots, N:$

\begin{align}
 {\mathbb E}   \Big\{   {\cal H} &(t,x^o(t),  \psi^o(t),Q^o(t),u_t^{1,o}, \ldots, u_t^{i-1,o}, u^i, u_t^{i+1,o}, \ldots, u_t^{N,o})   |{\cal G}_{0, t}^i \Big\}  \nonumber \\
  \geq & {\mathbb E} \Big\{   {\cal H}(t,x^o(t),  \psi^o(t),Q^o(t),u_t^{1,o}, \ldots, u_t^{i-1,o}, u_t^{i,o}, u_t^{i+1,o}, \ldots, u_t^{N,o})   |{\cal G}_{0, t}^i \Big\} ,  \nonumber \\
\hst \hst &  \forall u^i \in {\mathbb A}^i, \hso a.e. t \in [0,T], \hso {\mathbb P}|_{{\cal G}_{0,t}^i}- a.s., \hso i=1,2,\ldots, N.   \label{h4i}
\end{align}

\item[(2)] {\bf Sufficient Conditions.}  Under global convexity of the Hamiltonian with respect to the state and control variables  and convexity of the terminal pay-off function $\varphi(\cdot)$ the pair  $\{x^o(\cdot), u^o(\cdot)\}$ is optimal if it satisfies (\ref{h4i}).
\end{description}
An important feature obtained  during the derivation  is that the optimality conditions for a team optimal strategy   are  equivalent to the  optimality conditions for a person-by-person optimal strategy. This follows from Theorem~\ref{theorem5.1} and Corollary~\ref{corollary5.1}. \\

The  point to be made regarding the  derivation of the above optimality conditions, is that  we convert the problem into a centralized problem with the associated Hamiltonian system of equations to capture the constraints, and only at the final step, the optimality of decentralized strategies is addressed, by  identifying the conditional variational Hamiltonian which is consistent with the decentralized information structures. That is,  the Hamiltonian system (\ref{h2i}), (\ref{h3i}) is the one corresponding to centralized strategies, while the conditional Hamiltonian (\ref{h4i}) is  the projection of the centralized Hamiltonian onto the  subspace generated by the decentralized information structures.\\

\begin{comment}
\noi The case of decentralized partial noisy information structures
generated by noisy  nonlinear functions of any subset of the subsystem states $x^1, \ldots, x^N$ will be treated elsewhere.\\
\end{comment}

We conclude  the preliminary discussion  on classical optimization theory of centralized strategies versus decentralized strategies, by stating that there are no limitations in applying classical theory of optimization to distributed systems with  decentralized information structures. Rather,  the challenge is in the computation of the conditional   Hamiltonians,  and hence  the optimal strategies.  However, this has also remained  a challenge  for centralized fully or partially observed strategies. \\

\noi The specific objectives of this paper  are the following.

\begin{description}
\item[(a)] Derive  team games necessary conditions of optimality (stochastic maximum principle) for distributed stochastic  differential systems with decentralized information structures.

\item[(b)] Introduce assumptions so that the team games necessary conditions of optimality in {\bf (a)} are also sufficient;

\item[(c)] Derive person-by-person optimality conditions and discuss their relation with team optimality conditions;

\item[(d)] Prove  existence of optimal team and person-by-person strategies for distributed stochastic differential systems with decentralized information structures, using the theory of  relaxed  control strategies, and relate {\bf (a), (b), (c)} to regular decision strategies.

\end{description}
\ \
A detailed investigation of applications of the results of this part to specific linear and nonlinear distributed stochastic differential decision systems is  discussed in the second part of this two-part paper \cite{charalambous-ahmedFIS_Partii2012}  where we  derive the explicit expressions for the optimal decentralized  strategies.\\

 The rest of the paper is organized as follows. In Section~\ref{formulation}  we  formulate the distributed stochastic differential system with decentralized information structures.  In Section~\ref{existence}, we consider the question of  existence of optimal relaxed controls
(decisions). In Section~\ref{smp},  we develop the stochastic optimality conditions for team games with decentralized information structures,  consisting of necessary and sufficient conditions of optimality.  In Section~\ref{regular},  we  specialize the necessary and sufficient optimality conditions  to regular  strategies and  obtain corresponding  necessary and sufficient optimality conditions.  The  paper is concluded with some comments on possible extensions of our results.

 \section{Team Games of Stochastic Differential Systems }
 \label{formulation}

\noi In this section we introduce  the mathematical formulation of distributed stochastic  systems, the  information structures available to the decision makers for their actions,  and the definitions of collaborative decisions via team game optimality and person-by-person optimality. Throughout the terms "decision maker" or "control" are used interchangeably. A stochastic dynamical decision or control system is called distributed if it consists of an interconnection of at least two subsystems and decision makers. The underlying assumption for these distributed systems is that the decision makers actions are based on decentralized information structures. However, the decision makers are allowed to  exchange  information on their  law or strategy deployed,  e.g., the functional form of their strategies but not their actions. \\

 \noi {\bf  Some Basic  Terminologies }\\
   \begin{eqnarray*}
  &\mbox{DM}& \mbox{Abbreviation for "Decision Maker"}\\
 &{\mathbb Z}_N  \tri \{1,2,\ldots, N\} & \mbox{subset of natural numbers}\\
%  &I \tri [0, T]  \subset {\mathbb R}  &     \mbox{time duration of optimization}\\
 &s\tri \{s^1, s^2, \ldots,\ldots, s^N\} & \mbox{set consisting of $N$ elements }\\
  & s^{-i}=s \setminus\{s^i\}, \hso s=(s^{-i},s^i)  & \mbox{set $s$ minus $\{s^i\}$  }\\
 &{\cal L}({\cal X},{\cal Y})   & \mbox{ linear transformation mapping a vector space ${\cal X}$} \\   & & \mbox{ into a vector space ${\cal Y}$} \\
 & A^{(i)} & \mbox{$i$th column of a map $A \in  {\cal L}({\mathbb R}^n,{\mathbb R }^m)$, $i=1, \ldots, n$}\\
 & ({\mathbb A}^i,d)& \mbox{separable metric space for player $i \in {\mathbb Z}_N$ actions}\\
&{\mathbb A}^{(N)} \tri \times_{i=1}^N {\mathbb A}^i&    \mbox{product action space of $N$ players}  \\
&{\mathbb U}_{reg}^{i} [0,T]&     \mbox{regular admissible strategy of player $i \in {\mathbb Z}_N$} \\
&{\mathbb U}_{rel}^i[0,T]&    \mbox{relaxed admissible strategy of player $i \in {\mathbb Z}_N$} \\
%&u^{-i} \tri \{u \setminus u^i:   u \in  {\mathbb U}^{(N)}[0,T], u^{i} \in {\mathbb U}^{i}[0,T] \} & \mbox{strategies of players $i=1, 2, i-1,i+1, \ldots, n$},  %i=1,2, \ldots, n
   \end{eqnarray*}

\noi
  Let $\Big(\Omega,{\mathbb F},  \{ {\mathbb F}_{0,t}:   t \in [0, T]\}, {\mathbb P}\Big)$ denote a complete filtered probability space satisfying the usual conditions  \cite{liptser-shiryayev1977}, that is,  $(\Omega,{\mathbb F}, {\mathbb P})$ is complete, ${\mathbb F}_{0,0}$ contains all ${\mathbb P}$-null sets in ${\mathbb F}$. Note that filtrations $\{{\mathbb F}_{0,t} : t \in [0, T]\}$ are monotone in the sense that ${\mathbb F}_{0,s} \subseteq {\mathbb F}_{0,t}$, $\forall 0\leq s \leq t \leq T$. Moreover,   $\{ {\mathbb F}_{0,t}: t \in [0, T] \}$ is called  right continuous if ${\mathbb F}_{0,t} = {\mathbb F}_{0,t+} \tri \bigcap_{s>t} {\mathbb F}_{0,s}, \forall t \in [0,T)$ and it is called left continuous if  ${\mathbb F}_{0,t} = {\mathbb F}_{0,t-} \tri \sigma\Big( \bigcup_{s<t} {\mathbb F}_{0,s}\Big), \forall t \in (0,T]$. Throughout the paper  filtrations are denoted  by ${\mathbb F}_T \tri  \{ {\mathbb F}_{0,t}:   t \in [0, T]\}$, and they are  assumed to be right continuous and complete. \\

Consider a random process $\{z(t):  t \in [0,T]\}$ defined   on  the filtered probability space $(\Omega,{\mathbb F},\{ {\mathbb F}_{0,t}: t \in [0,T]\}, {\mathbb P})$ and  taking values in a metric space  $({\mathbb Z}, d).$
  The process $\{z(t): t \in [0,T]\}$ is said to be measurable  if the map $(t,\omega)  \rar  z(t,\omega)$ is  ${\cal B }([0,T]) \times{\mathbb  F}/{\cal B}({\mathbb Z})-$measurable where ${\cal B}(\mathbb Z)$ denotes the  Borel algebra of subsets of ${\mathbb Z}.$
  The process $\{z(t): t \in [0,T]\}$ is said to be $\{{\mathbb F}_{0,t}: t \in [0,T]\}-$adapted if for all $t \in [0,T]$, the map $\omega \rar z(t,\omega)$ is  ${\mathbb F}_{0,t}/{\cal B}({\mathbb Z})-$measurable.  The process $\{z(t): t \in [0,T]\}$ is said to be $\{{\mathbb F}_{0,t}: t \in [0,T]\}-$progresively measurable if for all $t \in [0,T]$, the map $(s,\omega) \rar z(s,\omega)$ is  $ {\cal B}([0,t]) \otimes {\mathbb F}_{0,t}/{\cal B}({\mathbb Z})-$measurable.  It can be shown that any stochastic process  $\{z(t):  t \in [0,T]\}$  on a filtered probability space  $(\Omega,{\mathbb F},\{ {\mathbb F}_{0,t}: t \in [0,T]\}, {\mathbb P})$ which is measurable and adapted has a progressively measurable modification \cite{liptser-shiryayev1977}. Unless otherwise specified, we shall say a process $\{z(t):  t \in [0,T]\}$ is $\{{\mathbb F}_{0,t}: t \in [0,T]\}-$adapted if the processes is $\{{\mathbb F}_{0,t}: t \in [0,T]\}-$progressively measurable.\\

In our derivations we make extensive use of the following spaces considered  by the authors in \cite{ahmed-charalambous2012a}. Let
 $L_{{\mathbb F}_T}^2([0,T],{\mathbb R}^n) \subset   L^2( \Omega \times [0,T], d{\mathbb P}\times dt,  {\mathbb R}^n) \equiv L^2([0,T], L^2(\Omega, {\mathbb R}^n)) $ denote the space of ${\mathbb F}_T-$adapted random processes $\{z(t): t \in [0,T]\}$   such that
\bes
{\mathbb  E}\int_{[0,T]} |z(t)|_{{\mathbb R}^n}^2 dt < \infty,
\ees
 which is a sub-Hilbert space of     $L^2([0,T], L^2(\Omega, {\mathbb R}^n))$.
  Similarly, let $L_{{\mathbb F}_T}^2([0,T],  {\cal L}({\mathbb R}^m,{\mathbb R}^n)) \subset L^2([0,T] , L^2(\Omega, {\cal L}({\mathbb R}^m,{\mathbb R}^n)))$ denote the space of  ${\mathbb  F}_{T}-$adapted $n\times m$ matrix valued random processes $\{ \Sigma(t): t \in [0,T]\}$ such that
 \bes
  {\mathbb  E}\int_{[0,T]} |\Sigma(t)|_{{\cal L}({\mathbb R}^m,{\mathbb R}^n)}^2 dt  \tri  {\mathbb E} \int_{[0,T]} tr(\Sigma^*(t)\Sigma(t)) dt < \infty.
  \ees

  \subsection{Regular Strategies}
\label{deterministic}
In this subsection we  consider  measurable vector valued functions, also known as regular  strategies.
% It is well-known that there are two formulation of the filtered probability space. The strong formulation,  in which the probability measure is fixed \'a priori and it is  independent of the decisions, while the information structures are modeled by $\sigma$-algebras that depend on the decisions, and the  weak formulation in which the probability measure depends on the decisions, while the information structures are fixed \'a priori.
We consider the strong formulation.
  Let $\Big(\Omega,{\mathbb F},\{ {\mathbb F}_{0,t}:   t \in [0, T]\}, {\mathbb P}\Big)$ denote a fixed complete filtered probability space on which are based  all random  processes considered in the paper. At this stage we do not specify how $\{{\mathbb F}_{0,t}: t \in [0,T]\}$ came about, but we require that Brownian motions are adapted to this filtration.\\

  \noi {\bf Admissible Decision Maker  Strategies}\\

 The Decision Makers (DM) $\{u^i: i \in {\mathbb Z}_N\}$ take values in a closed  convex subset of  linear  metric spaces $\{({\mathbb M}^i,d): i \in {\mathbb Z}_N\}$.  Let  ${\cal G}_{T}^i \tri \{  {\cal G}_{0,t}^i: t \in [0, T]\} \subset \{{\mathbb F}_{0,t}: t \in [0,T]\}$ denote the information available to DM $i$, $ \forall i \in {\mathbb Z}_N$.
     The admissible set  of regular strategies is  defined by
 \begin{align}
 {\mathbb U}_{reg}^i[0, T] \tri \Big\{   u^i   \in  L_{{\cal G}_T^i}^2([0,T],{\mathbb R}^{d_i})  : \:   u_t^i \in {\mathbb A}^i \subset {\mathbb R}^{d_i}, \: a.e. t \in [0,T], \: {\mathbb P}-a.s. \Big\}, \hso  \forall  i \in {\mathbb Z}_N. \label{cs1a}
     \end{align}
Clearly,   ${\mathbb U}_{reg}^i[0, T]$ is a  closed convex subset of  $L_{{\mathbb F}_T}^2([0,T],{\mathbb R}^{d_i})$, for $i=1,2, \ldots, N$.  That is,  $u^i  : [0,T] \times \Omega \rar {\mathbb A}^i$,  and
   $\{u_t^i: t  \in  [0, T] \}$
    is  ${\cal G}_T^i-$adapted, $\forall  i \in {\mathbb Z}_N$.\\
     An $N$ tuple of DM strategies    is by definition $(u^1,u^2, \ldots, u^N) \in {\mathbb U}_{reg}^{(N)}[0,T] \tri \times_{i=1}^N {\mathbb U}_{reg}^i[0,T]$,  which are nonanticipative  with respect to the information structures $\{  {\cal G}_{0,t}^i: t \in [0, T]\} , i=1,2, \ldots, N$.  Hence, the information structure of each DM, ${\cal G}_T^i$, is decentralized, and may be generated by local or global subsystem observables. Nonanticipative strategies are often utilized when deriving the minimum  principle for centralized stochastic control or decision systems \cite{yong-zhou1999}.  \\

 \noi {\bf Distributed Stochastic  Systems}\\

 Given a fixed  probability space $\Big(\Omega,{\mathbb F},\{ {\mathbb F}_{0,t}:   t \in [0, T]\}, {\mathbb P}\Big),$ a distributed stochastic system consists of an interconnection of $N$ subsystems.  Each subsystem $i$ has its own   state space ${\mathbb R}^{n_i}$, action space ${\mathbb A}^i \subset {\mathbb R}^{d_i}$, an exogenous   noise space ${\mathbb W}^i \tri {\mathbb R}^{m_i}$, and an  initial state $x^i(0)=x_0^i$, identified by the following quantities.
\begin{description}
\item[(S1)] $x^i(0)=x_0^i$:  an ${\mathbb R}^{n_i}$-valued  Random Variable;

\item[(S2)] $\{ W^i(t): t \in [0,T]\}$: an  ${\mathbb R}^{m_i}$-valued  standard Brownian motion which models the exogenous state noise, adapted to ${\mathbb F}_T$, independent of $x^i(0)$.
\end{description}
Each subsystem is described by a finite dimensional system of coupled stochastic differential equations of It\^o type as follows.
 \begin{align}
 dx^i(t) =&  f^i(t,x^i(t),u_t^i) dt  +\sigma^i(t,x^i(t),u_t^i)dW^i(t)
 + \sum_{j=1, j \neq i}^N f^{ij}(t,x^j(t),u_t^j)dt  \nonumber \\
 &+\sum_{j=1, j \neq i}^N \sigma^{ij}(t,x^j(t),u_t^j)dW^j(t) , \hst x^i(0) = x_0^i, \hso  t \in (0,T], \hso i \in {\mathbb Z}_N. \label{eq1ds}
 \end{align}
On the product space $({\mathbb X}^{(N)}, {\mathbb A}^{(N)}, {\mathbb W}^{(N)})$, where ${\mathbb X}^{(N)}\tri \times_{i=1}^N {\mathbb R}^{n_i} , {\mathbb A}^{(N)} \tri \times_{i=1}^N {\mathbb A}^i,  {\mathbb W}^{(N)} \tri \times_{i=1}^N {\mathbb R}^{m_i}$, one
defines the augmented vectors by
\bes
 W \tri (W^1, W^2, \ldots, W^{N}) \in {\mathbb R}^m, \hso  u \tri (u^1, u^2, \ldots, u^{N}) \in {\mathbb R}^d, \hso x \tri (x^1, x^2, \ldots, x^{N}) \in {\mathbb R}^n.
 \ees
Then  on the product space the distributed   system  is described in compact form by
 \begin{eqnarray}
 dx(t) =  f(t,x(t),u_t) dt + \sigma(t,x(t),u_t)~dW(t), \hst x(0) = x_0, \hst  t \in (0,T], \label{eq1}
 \end{eqnarray}
 where $f: [0,T] \times {\mathbb R}^n\times {\mathbb A}^{(N)} \longrightarrow {\mathbb R}^n$ denotes the drift and $\sigma : [0,T] \times {\mathbb R}^n\times {\mathbb  A}^{(N)} \longrightarrow {\cal L}({\mathbb R}^m, {\mathbb R}^n)$ the diffusion coefficients. Note that (\ref{eq1}) is very general since no specific interconnection structure is assumed among the different subsystems. \\

%\noi{\bf Nonanticipative and Feedback Information Structures}\\

   \noi {\bf Pay-off Functional}\\

Consider the distributed system (\ref{eq1}) with decentralized full information structures.
  \noi Given a $u \in {\mathbb U}_{reg}^{(N)}[0,T],$ we  define the reward or performance criterion by
 \begin{align}
 J(u) \equiv  J(u^1,u^2,\ldots,u^N)   \tri
    {\mathbb E} \biggl\{   \int_{0}^{T}  \ell(t,x(t),u_t) dt + \varphi(x(T)\biggr\},            \label{cfd}
  \end{align}
  where $\ell: [0,T] \times {\mathbb R}^n\times {\mathbb U}^{(N)} \longrightarrow (-\infty, \infty]$ denotes the integrand for the running cost functional  and $\varphi : {\mathbb R}^n \longrightarrow (-\infty, \infty]$,  the terminal cost function.
 Notice that the performance of the decentralized system is measured  by  a single pay-off functional.  The interpretation is that there is a centralized layer where  the quality of  individual   decision makers strategies are  evaluated for  a common goal. Therefore, the underlying assumption concerning the single pay-off instead of  multiple pay-offs  (one for each decision maker) is that the team objective   can be met.

\noi For deterministic as well as  stochastic systems, it is well known that   if the set  ${\mathbb A}^i$ is  not convex, there may not exist any optimal control.  For this reason it is necessary to introduce relaxed strategies as discussed  in the next subsection. \\

\subsection{Relaxed Strategies}
\label{relaxed}
This paper will focus on   relaxed strategies (also called randomized strategies) and later on specialize to  regular strategies (measurable functions).    Therefore, we introduce the formulation  based on  relaxed  strategies (e.g. probability  measures on the action space). \\

 \noi{\bf Distributed Stochastic  Systems}\\

 For each $i \in {\mathbb Z}_N$, let $({\mathbb M}^i,d)$  be a separable metric space with  ${\mathbb A}^i \subset {\mathbb M}^i$ compact, and let  ${\cal B}({\mathbb A}^i)$ denote the Borel subsets of ${\mathbb A}^i$. Let $C({\mathbb A}^i)$ denote the space of continuous functions on ${\mathbb A}^i$.   Let    ${\cal M}({\mathbb A}^i)$ denote the space of regular bounded  signed Borel measures   on ${\cal B}({\mathbb A}^i)$  and ${\cal M}_1({\mathbb A}^i) \subset {\cal M}({\mathbb A}^i)$ the space of regular probability measures.  The  DM strategies with different information structures  on the  time  interval $[0, T]$ will be described  through the topological dual of the Banach space   $L_{{\cal G}_T^i}^1([0, T],C({\mathbb A}^i))$, the $L^1$-space of ${\cal G}_T^i \tri \{{\cal G}_{0,t}^i: t \in [0,T]]\}-$ adapted $C({\mathbb A}^i)$ valued functions, for $i \in {\mathbb Z}_N$.   For each $i \in {\mathbb Z}_N$ the  dual of this space is given by  $L_{ {\cal G}_T^i}^{\infty}([0, T],{\cal M}({\mathbb A}^i))$ which  consists of weak$^*$  measurable  ${\cal G}_T^i$ adapted ${\cal M}({\mathbb  A}^i)$ valued functions.  The DM (control) strategies  are drawn from   the subspace $L_{{\cal G}_T^i}^{\infty}([0, T],{\cal M}_1({\mathbb A}^i)) \subset    L_{{\cal G}_T^i}^{\infty}([0, T],{\cal M}({\mathbb A}^i)).$  For convenience notation we denote this by
  \begin{align}
  {\mathbb U}_{rel}^i[0,T] \tri      L_{{\cal G}_T^i}^{\infty}([0, T],{\cal M}_1({\mathbb A}^i)), \hst   i \in {\mathbb Z}_N, \label{rc1}
     \end{align} 
     and the team strategies by the product space
     \bes
          {\mathbb U}_{rel}^{(N)}[0,T] \tri \times_{i=1}^N {\mathbb U}_{rel}^i[0,T] , \hst {\cal M}_1({\mathbb A}^{N)})\tri \times_{i=1}^N {\cal M}_1({\mathbb A}^i).
           \ees

\noi Thus, for any $i  \in {\mathbb Z}_N$, given the information ${\cal G}_{ T}^i$, player  $ \{u_t^i: t \in [0, T]\}$ is a stochastic kernel (conditional distribution) defined by
\bes
u_t^i(\Gamma)=  q_t^i(\Gamma | {\cal G}_{0, t}^i), \hst  \mbox{for} \hso t \in [0, T], ~\mbox{and }\hst \forall \Gamma \in {\cal B}({\mathbb A}^i).
\ees
Clearly,  for each $ i \in {\mathbb Z}_N$ and for  every $\varphi \in C({\mathbb A}^i)$ the process
\bes
\int_{{\mathbb A}^i} \varphi(\xi) u_t^i(d\xi) = \int_{{\mathbb A}^i} \varphi(\xi) q_t^i(d\xi | {\cal G}_{0,t}^i ), \hst t \in [0,T],
\ees
is ${\cal G}_{T}^i-$ progressively measurable. Given a $u \in {\mathbb U}_{rel}^{(N)}[0,T]$, the  distributed  system is  written in compact form as
\begin{align}
dx(t)  =  f(t,x(t),u_t)dt + \sigma(t,x(t),u_t) dW(t), \hst   x(0) = x_0, \hso t \in [0,T],  \label{sd1}
\end{align}
where the drift and diffusion coefficient is now  defined by
\begin{align}
F(t,x,u_t)  \tri  \int_{{\mathbb A}^{(N)}} \Big( b(t,x,\xi^1, \xi^2, \ldots, \xi^N)\Big)  \times_{i=1}^N u_t^i(d\xi^i) dt , \hst   t \in [0,T),\label{pd1}
\end{align} for $F = \{f, \sigma\}, $
\\

\noi {\bf Pay-off Functional}\\

Given a $ u \in {\mathbb U}_{rel}^{(N)}[0,T]$ the performance criterion is defined  by

 \begin{align}
 J(u) & \tri   {\mathbb E}\biggl\{\int_0^T      \ell(t,x(t),u_t) dt + \varphi(x(T))\biggr\}    \label{cf2} \\
& \equiv {\mathbb E}\biggl\{\int_0^T   \int_{{\mathbb A}^{(N)}}    \Big( \ell(t,x(t),\xi^1,\xi^2,\ldots, \xi^N) \Big) \times_{i=1}^N    u_t^i(d\xi^i)  dt + \varphi(x(T))\biggr\}   \label{cf1}
   \end{align}
where  $\ell$ and $\varphi$ are as defined before.

\subsection{Team and Person-by-Person Optimality}
\label{tpbp}
In this section we give the precise definitions of team and person-by-person (i.e., player-by-player) optimality for relaxed and regular strategies.  There are many possible information structures for control strategies  $\{u^i: i \in {\mathbb Z}_N\}$. We consider the following.\\

%\begin{description}

{\bf (NIS): Nonanticipative Information Structures.}  Decision   $u^i$ is adapted to the filtration
  ${\cal G}_T^i \subset {\mathbb F}_T$ which is  generated by the   $\sigma-$ algebra induced by  any combination of the subsystems Brownian motions and their increments $\{(W^1(t),W^2(t), \ldots, W^N(t)): t \in [0,T]\}, \forall i \in {\mathbb Z}_N$.
  This is often called open loop information, and  it is the one used in classical stochastic control with centralized full information  to derive the maximum principe \cite{yong-zhou1999}.\\

{\bf (FIS):  Feedback Information Structures.}  Decision $u^i$ is adapted
to the filtration ${\cal G}_T^{z^{i}}$ generated by the $\sigma-$algebra   ${\cal G}_{0,t}^{z^{i}} \tri  \sigma\{ z^{i}(s): 0 \leq s \leq t\}, t \in [0,T]$, where  the observables $z^i$ are  nonanticipative measurable functionals of any combination of the states defined  by
\bea
z^i(t) = h^i(t,x), \hst h^i : [0,T] \times C([0,T], {\mathbb R}^n)  \longrightarrow {\mathbb R}^{k_i}, \hso i \in {\mathbb Z}_N. \label{eq1aa}
\eea Note that  the state $x$ and hence the observables $z^i$ may depend on controls.  \\
The set of admissible regular  feedback strategies is defined by
\bea
{\mathbb U}_{reg}^{(N),z}[0,T] \tri \Big\{u \in  {\mathbb U}_{reg}^{(N)}[0,T]: u_t^i \hso \mbox{is} \hso {\cal G}_{0,t}^{z^{i}}-\mbox{measurable},  t \in [0,T], \hso i=1,\ldots, N\Big\}. \label{fstrategies1}
\eea
 Similarly,  the set of admissible relaxed feedback  strategies is defined by
\bea
{\mathbb U}_{rel}^{(N),z}[0,T] \tri \Big\{ u \in  {\mathbb U}_{rel}^{(N)}[0,T]: u^i \in L_{ {\cal G}_{T}^{z^{i}}}^\infty([0,T], {\cal M}_1({\mathbb A^i})), \hso i=1,\ldots, N\Big\}. \label{fstrategies2}
\eea

 \noi One might be tempted to believe that nonanticipative strategies  might be restrictive, because they are not explicitly described in terms of feedback.  We will show that this is not true. In fact  such strategies cover a large number of interesting problems. \\

 \begin{problem}(Team Optimality)
 \label{problem1}\\

{\bf  (RS): Relaxed Strategies.} Given the  pay-off functional (\ref{cf2}),   constraint (\ref{sd1})   the  $N$ tuple of relaxed strategies   $u^o \tri (u^{1,o}, u^{2,o}, \ldots, u^{N,o}) \in {\mathbb U}_{rel}^{(N)}[0,T]$  is called nonanticipative team optimal if it satisfies
 \bea
 J(u^{1,o}, u^{2,o}, \ldots, u^{N,o}) \leq J(u^1, u^2, \ldots, u^N),  \hst \forall  u\tri (u^1, u^2, \ldots, u^N) \in {\mathbb U}_{rel}^{(N)}[0,T] \label{cfd1a}
 \eea
 Any $u^o   \in {\mathbb U}_{rel}^{(N)}[0,T]$ satisfying (\ref{cfd1a})
is called an optimal  relaxed decision strategy (or control) and the corresponding $x^o(\cdot)\equiv x(\cdot; u^o(\cdot))$ (satisfying (\ref{sd1})) the  optimal state process.\\
Similarly,  feedback  team optimal strategies  are defined with respect to  $u^o  \in {\mathbb U}_{rel}^{(N),z}[0,T]$ %$, u^o  \in {\mathbb U}_{rel}^{(N),I^u}[0,T]$, respectively.
\\

{\bf (NRS): Regular Strategies.} Regular nonanticipative team
optimal strategies are  defined with respect to pay-off (\ref{cfd}), constraint (\ref{eq1}),    and    $u^o \in {\mathbb U}_{reg}^{(N)}[0,T]$,  while  feedback team optimal  strategies are defined with respect to    $u^o \in {\mathbb U}_{reg}^{(N),z}[0,T]$. % and $u^o \in {\mathbb U}_{rel}^{(N),I^u}[0,T]$, respectively.

  \end{problem}

 By definition, Problem~\ref{problem1}  is a dynamic team problem with each DM having a different information structure (decentralized). To the best of the authors knowledge there seems to have been  no attempt in the literature to address the  Problem~\ref{problem1}.   An alternative approach to handle such problems with decentralized information structures is to restrict the definition of optimality to   the so-called person-by-person (player-by-player)  equilibrium.  \\
 Define
 \bes
 \tilde{J}(v,u^{-i}) \tri J(u^1,u^2,\ldots, u^{i-1},v,u^{i+1},\ldots,u^N)
 \ees

 \begin{problem}(Person-by-Person Optimality)
 \label{problem2}\\

{\bf (RS): Relaxed Strategies.} Given the   pay-off functional (\ref{cf2}),   constraint (\ref{sd1})   the  $N$ tuple of relaxed strategies   $u^o \tri (u^{1,o}, u^{2,o}, \ldots, u^{N,o}) \in {\mathbb U}_{rel}^{(N)}[0,T]$  is called nonanticipative person-by-person optimal  if it satisfies
\begin{align}
 \tilde{J}(u^{i,o}, u^{-i,o}) = J(u^o) \leq \tilde{J}(u^{i}, u^{-i,o}), \hst \forall u^i \in {\mathbb  U}_{reg}^i[0,T], \hso \forall i \in {\mathbb Z}_N. \label{cfd2}
 \end{align}
 Similarly,  feedback  person-by-person optimal strategies are defined with respect to  $u^o \in {\mathbb U}_{rel}^{(N),z}[0,T]$. % and $u^o \in {\mathbb U}_{rel}^{(N),I^u}[0,T]$, respectively.
 \\

{\bf (NRS): Regular Strategies.} Regular nonanticipative person-by-person
optimal strategies are   defined with respect to pay-off (\ref{cfd}), constraint (\ref{eq1}),    and    $u^o \in {\mathbb U}_{reg}^{(N)}[0,T]$,  while  feedback  person-by-person optimal  strategies are defined with respect to $u^o \in {\mathbb U}_{reg}^{(N),z}[0,T]$. % and $u^o \in {\mathbb U}_{reg}^{(N),I^u}[0,T]$,respectively.
  \end{problem}

The interpretation of (\ref{cfd2}) is that the  variation and hence evaluation (of team optimality)  is done by the central layer and it  is  this layer alone that can determine if the decision for the $i$-th player  is optimal or not.  Even for Problem~\ref{problem2}  the  authors of this paper are not aware of any publication  which addresses necessary and/or sufficient conditions of optimality.  Conditions (\ref{cfd2}) are analogous to the Nash equilibrium strategies of team games consisting of a single pay-off and $N$ DM.  The person-by-person optimal strategy states that none of the $N$  members (possibly  with different  information structures) can deviate unilaterally from the optimal strategy and gain by doing so. The rationale for the restriction to person-by-person optimal strategy is based on the fact that the actions of the $N$ DM
are not communicated to each other, and hence they cannot
do better than restricting attention to this optimal strategy.\\

Problems~\ref{problem1}, \ref{problem2} using  relaxed strategies are  the main problems addressed in this paper, while conclusions for regular strategies are drawn  from these results. Clearly,  any strategy which  is optimal for Problem~\ref{problem1} is also a person-by-person optimal and hence optimal for  Problem~\ref{problem2}. \\

   \section{Existence of  Team Optimal Strategies}
  \label{existence}
As mentioned earlier, not every control  problem admits optimal  regular strategies. However, in many problems relaxed strategies exist under certain mild  assumptions.  In this section we use a similar procedure as the one developed in
   \cite{ahmed-charalambous2012a} for centralized information structures to prove (i) existence of solution of the distributed stochastic  dynamical decision system (\ref{sd1}), and (ii) existence of optimal relaxed strategies  for the Problem~\ref{problem1}. \\

%\noi  For the   class of admissible relaxed strategies ${\mathbb U}_{rel}^{(N)}[0,T] \tri \times_{i=1}^n L_{{\cal G}_T^i}^{\infty}([0,T],{\cal M}_1({\mathbb A}^i))$,
% the set of test functions corresponding to the $ith$ strategy  is the space given %by
% $L_{{\cal G}_T^i}^{1}([0, T],C({\mathbb A}^i)) , \forall  i \in {\mathbb Z}_N$.
A generalized  sequence $u^{i,\alpha} \in {\mathbb  U}_{rel}^i[0,T]$ is said to converge (in the weak$^*$ topology or)  vaguely  to $u^{i,o},$ written $ u^{i,\alpha}\buildrel v\over\longrightarrow u^{i,o}$, if and only if  for every  $\varphi \in L_{ {\cal G}_T^i}^1([0,T],C({\mathbb A}^i))$
\bes
  {\mathbb  E} \int_{[0,T] \times {\mathbb A}^i}  \varphi_t(\xi) u_t^{i,\alpha}(d\xi) dt \longrightarrow  {\mathbb E} \int_{[0,T] \times {\mathbb A}^i}  \varphi_t(\xi) u_t^{i,o}(d\xi) dt \hst \mbox{as} \hso \alpha \rightarrow \infty, \hst \forall i\in {\mathbb Z}_N.
  \ees

\noi With respect to    the vague (weak$^*$)  topology the set ${\mathbb U}_{rel}^i[0,T]$ is  compact, and from here on we assume that ${\mathbb U}_{rel}^i[0, T], \forall i \in {\mathbb Z}_N$ has been endowed with this vague topology. \\

 Let $B_{{\mathbb F}_T}^{\infty}([0,T], L^2(\Omega,{\mathbb R}^n))$ denote the space of ${\mathbb F}_T$-adapted ${\mathbb R}^n$ valued second order random processes endowed with the norm topology  $\parallel  \cdot \parallel$ defined by
\bes
 \parallel x\parallel^2  \tri \sup_{t \in [0,T]}  {\mathbb E}|x(t)|_{{\mathbb R}^n}^2.
 \ees
To  study the question of existence of solution to ({\ref{sd1}) we use the following assumptions.

\begin{assumptions}
\label{A1-A4}
The drift $f$ and diffusion coefficients $\sigma$ associated with (\ref{sd1}) are defined by the    Borel measurable  maps:
\begin{eqnarray*}
 f: [0,T] \times {\mathbb R}^n \times {\mathbb A}^{(N)} \longrightarrow {\mathbb R}^n , \hst  \sigma: [0,T] \times {\mathbb R}^n \times {\mathbb A}^{(N)} \longrightarrow {\cal L}({\mathbb R}^m, {\mathbb R}^n)
 \end{eqnarray*} and they are continuous in the last two arguments and  assumed to satisfy  the following basic properties:.

\begin{description}
\item[(A0)] $({\mathbb A}^i,d), \forall i \in {\mathbb Z}_N$ are compact.
\end{description}

\noi There exists a $K \in L^{2,+}([0,T], {\mathbb R})$ such that

\begin{description}
\item[(A1)] $|f(t,x,\xi)-f(t,y,\xi)|_{{\mathbb R}^n} \leq K(t) |x-y|_{{\mathbb R}^n}$ uniformly in $\xi \in {\mathbb A}^{(N)}$;

\item[(A2)] $|f(t,x,\xi)|_{{\mathbb R}^n} \leq K(t) (1 + |x|_{{\mathbb R}^n})$ uniformly in $\xi \in {\mathbb A}^{(N)}$

\item[(A3)] $|\sigma(t,x,\xi)-\sigma(t,y,\xi)|_{{\cal L}({\mathbb R}^m, {\mathbb R}^n)} \leq K(t) |x-y|_{{\mathbb R}^n}$ uniformly in  $\xi \in {\mathbb A}^{(N)}$;

\item[(A4)] $|\sigma(t,x,\xi)|_{{\cal L}({\mathbb R}^m, {\mathbb R}^n)} \leq K(t) (1+ |x|_{{\mathbb R}^n})$  uniformly in $\xi \in {\mathbb A}^{(N)}$;

\item[(A5)] $f(t,x,\cdot), \sigma(t,x,\cdot)$ are continuous in $\xi \in {\mathbb A}^{(N)}$, $\forall (t,x) \in [0,T] \times {\mathbb R}^n$.
\end{description}
\end{assumptions}

\noi Assumptions~\ref{A1-A4}, {\bf (A1)-(A4)} are the so-called It\^o conditions for existence and uniqueness of strong solutions (having continuous sample paths) \cite{yong-zhou1999}.\\

\noi The following lemma proves  the  existence of solutions and their continuous dependence on the decision variables.

\begin{lemma}
\label{lemma3.1}
 Suppose Assumptions~\ref{A1-A4} hold. Then for any ${\mathbb  F}_{0,0}$-measurable initial state $x_0$ having finite second moment, and any $u \in {\mathbb U}_{rel}^{(N)}[0,T]$,  the following hold.

\begin{description}

 \item[(1)]  System (\ref{sd1}) has a unique solution   $x \in B_{{\mathbb F}_T}^{\infty}([0,T],L^2(\Omega,{\mathbb R}^n))$  having a continuous modification, that is, $x \in C([0,T],{\mathbb R}^n)$,  ${\mathbb P}-$a.s, $\forall i \in {\mathbb Z}_N$.

\item[(2)]  The solution of  system  (\ref{sd1}) is continuously dependent on the control, in the sense that, as $u^{i, \alpha} \buildrel v \over\longrightarrow u^{i,o}$  in ${\mathbb U}_{rel}^i[0,T]$, $\forall i \in {\mathbb Z}_N$,  $x^\alpha \buildrel s \over\longrightarrow x^o $ in $B_{{\mathbb F}_T}^{\infty}([0,T],L^2(\Omega,{\mathbb R}^n)), \forall i \in {\mathbb Z}_N$.
\end{description}
These statements also hold for feedback  strategies $u \in {\mathbb U}_{rel}^{(N),z^u}[0,T]$. %, u \in {\mathbb U}_{rel}^{(N),I^u}[0,T]$, respectively.

\end{lemma}

\begin{proof} 	Since the class of policies ${\mathbb U}_{rel}^i[0,T]$, $\forall i \in {\mathbb Z}_N$ is compact in the vague topology, then $\times_{i=1}^N {\mathbb U}_{rel}^i[0,T]$ is also compact in this topology. Utilizing this observation the  proof is identical to that of  \cite{ahmed-charalambous2012a}, Lemma~3.1.

\end{proof}

\noi Using the  results of Lemma~\ref{lemma3.1} in the next theorem we establish existence of  a minimizer $u^o \in {\mathbb U}_{rel}^{(N)}[0,T]$ for Problem~\ref{problem1}.  We need  the following   assumptions.

\begin{assumptions}
\label{assumptionscost}
The functions $\ell$ and  $\varphi$ associated with the  pay-off (\ref{cf2}) are  Borel measurable maps:
\bes
\ell: [0,T] \times {\mathbb R}^n\times {\mathbb A}^{(N)} \longrightarrow (-\infty,+\infty], \hst  \varphi:{\mathbb R}^n \longrightarrow (-\infty,+\infty].
\ees
satisfying  the following basic conditions:

\begin{description}
\item[(B1)] $x\longrightarrow \ell(t,x,\xi)$ is  continuous on ${\mathbb R}^n$ for each $t\in [0,T]$, uniformly with respect to $\xi \in {\mathbb A}^{(N)}$;

\item[(B2)] $\exists$ $h \in L_1^+([0,T], {\mathbb R})$ such that for each $t \in [0,T]$,  $|\ell(t,x,\xi)| \leq h(t) (1 + |x|_{{\mathbb R}^n}^2)$;

\item[(B3)] $x \longrightarrow \varphi(x)$ is lower semicontinuous on ${\mathbb R}^n$ and  $\exists$ $c_0,c_1\geq 0$ such that $|\varphi(x)| \leq c_0 + c_1 |x|_{{\mathbb R}^n}^2.$
\end{description}
\end{assumptions}

\noi Now we present  the following existence theorem \cite{ahmed-charalambous2012a}.

\begin{theorem}(Existence of Team Optimal Strategies)
\label{theorem3.2}
Consider Problem~\ref{problem1} and suppose Assumptions~\ref{A1-A4}  and \ref{assumptionscost} hold.  Then there exists a  team  decision $u^o \tri (u^{1,o},u^{2,o}, \ldots, u^{N,o}) \in {\mathbb  U}_{rel}^{(N)}[0,T]$ at  which  $J(u^1, u^2, \ldots, u^N)$ attains its infimum.
Existence also holds for  $u^o  \in {\mathbb  U}_{rel}^{(N),z^u}[0,T]$. % and $u^o  \in {\mathbb  U}_{rel}^{(N),I^u}[0,T]$.

\end{theorem}

 \begin{proof}  Since the class of control policies  ${\mathbb U}_{rel}^N[0,T]$ %and ${\mathbb  U}_{rel}^{(N),I^u}[0,T]$
 is compact in the vague  topology, it suffices to prove that $J(\cdot)$ is lower semicontinuous with respect to this topology.  This  follows precisely  from the same procedure as in \cite{ahmed-charalambous2012a}, Theorem~3.2.
 \end{proof}

We conclude this section by stating that existence of team optimal strategies utilizing decentralized information structures follows directly from analogous results of centralized stochastic control strategies \cite{ahmed-charalambous2007}.

\section{Optimality Conditions for Relaxed Strategies }
\label{smp}
In this section we present the  necessary and sufficient conditions of  optimality  for the team game of Problem~\ref{problem1}.
The derivation of stochastic minimum principle (necessary conditions of optimality) or stochastic Pontryagin's minimum principle  is  based on the martingale representation approach.
For this reason we shall fisrt state   certain fundamental   properties of semi martingales,  which are used  in the derivation.  \\

\begin{definition}
\label{definition4.2}
Let  ${\mathbb F}_T$ denote a complete filtration generated by an ${\mathbb R}^m-$dimensional Brownian motion process $\{W(t): t \in [0,T]\}$.  An ${\mathbb R}^n-$valued  random process $\{m(t): t \in [0,T]\}$ is said to be a square integrable continuous   ${\mathbb F}_T-$semi martingale if and only if  it has a representation

\begin{eqnarray}
 m(t) = m(0) + \int_0^t v(s) ds + \int_0^t \Sigma(s) dW(s), \hst t \in [0,T], \label{eq12}
  \end{eqnarray}
   for some $v \in L_{{\mathbb F}_T}^2([0,T],{\mathbb R}^n)$ and $\Sigma \in L_{{\mathbb F}_T}^2([0,T],{\cal L}({\mathbb R}^m,{\mathbb R}^n))$  and for some ${\mathbb R}^n-$valued ${\mathbb  F}_{0,0}-$measurable  random variable  $m(0)$ having finite second moment. The set of all such semi martingales is denoted by ${\cal S M}^2[0,T]$.
\end{definition}

 \noi We need the  following class of ${\mathbb F}_T-$semi martingales:
   \begin{align}
    {\cal SM}_0^2[0,T] \tri \Big\{ m & \in {\cal S M}^2[0,T]:   m(t)   = \int_0^t v(s) ds + \int_0^t \Sigma(s) dW(s),  \hst   t \in [0,T], \nonumber \\
    &  \mbox{for} \hso v \in L_{{\mathbb F}_T}^2([0,T],{\mathbb R}^n) \hso \mbox{ and} \hso \Sigma \in L_{{\mathbb F}_T}^2([0,T],{\cal L}({\mathbb R}^m,{\mathbb R}^n)) \Big\}.
    \label{eq13}
    \end{align}

\noi Now we present a  fundamental result which is used  in the derivation of  minimum  principle.

\begin{theorem}(Semi martingale Representation)
\label{theorem4.3}
The class of semi martingales ${\cal SM}_0^2[0,T]$ is a real linear vector space and it is a  Hilbert space with respect to the norm topology   $\parallel m\parallel_{{\cal SM}_0^2[0,T]}$ given by
\bes
 \parallel m \parallel_{{\cal SM}^2_0[0,T]}  \tri \Big({\mathbb  E}\int_{[0,T]} |v(t)|_{{\mathbb R}^n}^2 dt + {\mathbb  E} \int_{[0,T]} tr(\Sigma^*(t)\Sigma(t)) dt\Big)^{1/2}.
 \ees
Moreover,  the space ${\cal SM}_0^2[0,T]$ is isometrically isomorphic to the space $$L_{{\mathbb F}_T}^2([0,T],{\mathbb R}^n)\times L_{{\mathbb F}_T}^2([0,T],{\cal L}({\mathbb R}^m,{\mathbb R}^n)).$$

% denoted as $${\cal SM}_0^2[0,T] \cong L_{{\mathbb F}_T}^2([0,T],{\mathbb %R}^n)\times L_{{\mathbb F}_T}^2([0,T],{\cal L}({\mathbb R}^m,{\mathbb R}^n)).%$$
\end{theorem}

\begin{proof}  For proof see Theorem 4.3 in  \cite{ahmed-charalambous2012a}.

\end{proof}

 For the  derivation of stochastic minimum principle of  optimality   we shall require stronger regularity conditions for the drift and diffusion coefficients  $\{ b,\sigma\}$, as well as, for the running and terminal pay-offs functions   $\{\ell,\varphi\}.$  These  are given below.

\begin{assumptions}
\label{NC1}
${\mathbb E} |x(0)|_{{\mathbb R}^n}^2 <\infty$ and the maps of $\{f,\sigma,\ell, \varphi\} $ satisfy the following conditions.

\begin{description}

\item[(C1)] The triple $\{f,\sigma,\ell\} $ are measurable in $t \in [0,T]$;

\item [(C2)] The quadruple   $\{f,\sigma,\ell, \varphi\} $ are  once continuously differentiable with respect to the state variable $x \in {\mathbb R}^n$;

\item[(C3)] The first derivatives of $\{f,\sigma\}$ with respect to the state are bounded uniformly on $[0,T] \times {\mathbb R}^n \times {\mathbb A}^{(N)}$.

\end{description}
\end{assumptions}

\noi Consider the Gateaux derivative of $\sigma$ with respect to the  variable  at the point $(t,z,\nu) \in [0,T] \times {\mathbb R}^n\times_{i=1}^N  {\cal M}_1({\mathbb A}
^i)$   in the direction $\eta \in {\mathbb R}^n$ defined by
\bes
   \sigma_x(t,z,\nu; \eta) \tri   \lim_{\varepsilon \rightarrow 0}\frac{1}{\varepsilon} \Big\{ \sigma(t,z + \varepsilon \eta, \nu)- \sigma(t,z,\nu)\Big\}, \hst t \in [0,T].
  \ees
    Note that the map $\eta \longrightarrow \sigma_x(t,z,\nu; \eta)$ is linear, and  it follows from Assumptions~ \ref{NC1}, {\bf (C3)} that  there exists a finite positive number $\beta>0$ such that
    \bes
     |\sigma_x(t,z,\nu; \eta)|_{{\cal L}({\mathbb R}^m,{\mathbb R}^n)} \leq \beta |\eta|_{{\mathbb R}^n},   \hst t \in  [0,T].
     \ees
\noi   In order to present the necessary conditions of optimality we need the so called variational equation.
Let us first introduce the variational equation for nonanticipative information structures. Suppose $u^o \tri (u^{1,o}, u^{2,o}, \ldots, u^{N,o}) \in {\mathbb U}_{rel}^{(N)}[0,T]$ denotes the optimal decision and $u \tri (u^1, u^2, \ldots, u^N) \in {\mathbb U}_{rel}^{(N)}[0,T]$ any other decision.  Since ${\mathbb U}_{rel}^i[0,T]$ is convex $\forall i \in {\mathbb Z}_N$, it is clear that  for any $\varepsilon \in [0,1]$,
\bes
 u_t^{i,\varepsilon} \tri u_t^{i,o} + \varepsilon (u_t^i-u_t^{i,o}) \in {\mathbb U}_{rel}^i[0,T], \hst \forall i \in {\mathbb Z}_N.
 \ees
 Let $x^{\varepsilon}(\cdot)\equiv x^\veps(\cdot; u^\veps(\cdot))$ and  $x^{o}(\cdot) \equiv x^o(\cdot;u^o(\cdot))  \in B_{{\mathbb F}_T}^{\infty}([0,T],L^2(\Omega,{\mathbb R}^n))$ denote the solutions  of the system equation (\ref{sd1})  corresponding to  $u^{\varepsilon}(\cdot)$ and $u^o(\cdot)$, respectively.  Consider the limit
 \bes
  Z(t) \tri \lim_{\varepsilon\downarrow 0}  \frac{1}{\veps} \Big\{x^{\varepsilon}(t)-x^o(t)\Big\} , \hst t \in [0,T].
  \ees
We have the following result characterizing the process $\{Z(t): t \in [0,T]\}$.

\begin{lemma}
\label{lemma4.1}
Suppose Assumptions~\ref{NC1} hold and consider nonanticipative strategies ${\mathbb U}_{rel}^{(N)}[0,T]$.  The process $\{Z(t): t \in [0,T]\}$ as defined above is an element of the Banach space  \\  $B_{{\mathbb F}_T}^{\infty}([0,T],L^2(\Omega,{\mathbb R}^n))$ and it  is the unique solution of the variational stochastic differential equation
 \begin{align}
 dZ(t) &= f_x(t,x^o(t),u_t^o)Z(t)dt + \sigma_x(t,x^o(t),u_t^o; Z(t))~dW(t) \nonumber \\
 &+ \sum_{i=1}^N f(t,x^o(t),u^{-i,o},u_t^i-u_t^{i,o})dt + \sum_{i=1}^N\sigma(t,x^o(t),u_t^{-i,o},u_t^i-u_t^{i,o})dW(t), \hst Z(0)=0. \label{eq9}
 \end{align}
 having a continuous modification.
  \end{lemma}

\begin{proof}   We closely follow the steps in [33].  Writing   the  system (\ref{sd1}) as an integral equation with solutions  $x^\varepsilon, x^o$ corresponding to controls $u^{\varepsilon},u^{o}$ respectively and taking the  difference  $x^\varepsilon(t)-x^o(t)$  and dividing by $\varepsilon$ and then letting  $\varepsilon \longrightarrow 0$, it can be shown that it converges  for all $ t \in [0,T], {\mathbb P}-a.s.$ to the solution of system (\ref{eq9}). Note that the system  (\ref{eq9}) is a linear stochastic differential equation in $Z$ with non homogeneous terms given by the sum of the last two terms.   Let $\{z(t): t \in [0,T]\}$ denote the solution of  its homogenous part  given by
\begin{align}
dz(t) = f_x(t,x^o(t),u_t^{o})z(t) dt  +   \sigma_x(t,x^o(t),u_t^{i,o}; z(t))dW(t), \hso
z(s) = \zeta , \hso t \in [s, T].
\end{align}
By Assumptions~\ref{NC1} and  Lemma~\ref{lemma3.1} this system has    a unique solution $\{z(t): t \in [s,T]\}$ given by
\bes
z(t) = \Psi(t,s) \zeta, \hst  t \in [s,T],
\ees
 where $\Psi(t,s),  t \in [s, T]$ is the random (${\mathbb F}_T
 -$adapted)  transition operator for the homogenous system. Since the derivatives of $f$ and $\sigma$ with respect to the state are uniformly bounded, the transition operator   $\Psi(t,s),  t \in [s, T]$ is uniformly  ${\mathbb P}-$a.s. bounded (with values in the space of $n\times n$ matrices). \\
 \noi   By Using the random transition operator $\Psi$ we can write the solution of the non homogenous stochastic differential equation (\ref{eq9}) as follows,
  \begin{eqnarray}
 Z(t) =\int_{0}^t \Psi(t,s)d\eta(s), \hst t \in [0,T],
  \label{eq10}
 \end{eqnarray}
 where $\{\eta(t):t \in [0,T]\}$ is the semi martingale  given by the following  Ito differential,
\begin{align}
 d\eta(t) = \sum_{i=1}^N &f(t,x^o(t),u^{-i,o},u_t^i-u_t^{i,o})dt  \nonumber \\
 &+  \sum_{i=1}^N\sigma(t,x^o(t),u^{-i,o},u_t^i-u_t^{i,o})~dW(t), \hst \eta(0) = 0, \hso t \in (0,T].\label{eq11}
 \end{align}
 Note that $ \{\eta(t): t \in [0,T] \}$ is a continuous  square integrable   ${\mathbb F}_T-$adapted  semi martingale.  The fact that it has continuous modification  follows directly from the  representation (\ref{eq10}) and the continuity  of the semi martingale $\{\eta(t) : t \in [0,T]\}$.
\end{proof}

%%% Start Remove

\begin{comment}
~~~~~~~~ I AM NOT COMFORTABLE  WITH THE FOLLOWING COMMENTS, SPECIALLY $u_x$ OF WHICH WE KNOW NOTHING OTHER THAN MEASURABILITY. DIFFERENTIABILITY !!!!(?)~~~~~~~~~~~~~~~~~~~~~~~~~~~~~~~~~~~~~~~~\\

\end{comment}

%%% End Remove

Clearly, the variational equation for nonanticipative strategies ${\mathbb U}_{rel}^{(N)}[0,T]$ is obtained as in centralized control strategies found in  \cite{ahmed-charalambous2012a}. Next, we discuss the variational equation for feedback information structures. For $u\in {\mathbb U}_{rel}^{(N),z^u}[0,T]$ the  variational equation will also involve derivatives of $u$ with respect to the  state trajectory $x$, since such strategies utilize feedback.
% Therefore, the right side of (\ref{eq9}) will also include the term $f_u u_x Z$ and the one  for $\sigma$,  hence one should introduce smoothness assumptions on $u$.
To avoid this technicality, we first address the question as to whether optimizing $J(u)$ over nonanticipative information structures is the same as optimizing $J(u)$ over feedback information structures. If this is the case then the variational equation for $u\in {\mathbb U}_{rel}^{(N),z^u}[0,T]$ will be that of $u\in {\mathbb U}_{rel}^{(N)}[0,T]$.   We shall require the following assumption.

\begin{assumptions}
\label{a-nf}
The following holds.
\begin{description}
\item[(E1)] The diffusion coefficient $\sigma$  is independent of $u$ and both   $\sigma(\cdot,\cdot)$ and $\sigma^{-1}(\cdot, \cdot)$  are  uniformly  bounded.
\end{description}
\end{assumptions}
\ \

\noi Under the (additional) Assumptions~\ref{a-nf} we can prove  the following theorem.

\begin{theorem}
\label{thm-nf}
Consider Problem~\ref{problem1} and suppose Assumptions~\ref{A1-A4} and  \ref{a-nf} hold.  Define the $\sigma-$algebras
\bes
{\cal F}_{0,t}^{x(0),W} \tri \sigma\{x(0), W(s): 0\leq s \leq t\}, \hst   {\cal F}_{0,t}^{x^u} \tri \sigma\{x^u(s): 0 \leq s \leq t\}, \hst \forall t \in [0,T].
\ees
\noi Then for all   $u \in {\mathbb U}_{rel}^{(N),x^{u}}[0,T]$ the two $\sigma$-algebras are equivalent  written as an equality,  ${\cal F}_{0,t}^{x(0),W} = {\cal F}_{0,t}^{x^u}, \forall t \in [0,T]$.
\end{theorem}

\begin{proof} Clearly, by Lemma~\ref{lemma3.1}, we have $  {\cal F}_{0,t}^{x^u} \subset   {\cal F}_{0,t}^{x(0),W}, \forall u \in {\mathbb U}_{rel}^{(N)}[0,T],  t \in [0,T]$. By use of   Assumptions~\ref{a-nf} one can  easily verify  that  ${\cal F}_{0,t}^{x(0),W} \subset  {\cal F}_{0,t}^{x^u}, \forall t \in [0,T].$ This  completes the proof.
\end{proof}

Under the conditions of Theorem~\ref{thm-nf},  for any stochastic kernel $\{u_t^i(\Gamma) \equiv q_t^i(\Gamma| {\cal G}_{0,t}^{x^{i,u}}): t \in [0,T]\}  \in {\mathbb U}_{rel}^{x^{i,u}}[0,T], \Gamma \in {\cal B}({\mathbb A}^i)$ which is ${\cal G}_{0,t}^{x^{i,u}}-$measurable there exists a function $\phi^i(\cdot)$ adapted to a sub-$\sigma-$algebra of ${\cal F}_{0,t}^i \subset {\cal F}_{0,t}^{x(0)
,W}$ such that $u_t^i(\Gamma)= q_t^i(\Gamma| \phi^i(t, x(0), W(\cdot \bigwedge t,\omega))), {\mathbb P}-a.s, \forall t \in [0,T], i=1,\ldots N$. \\Let ${\cal F}_T^i \tri \{ {\cal F}_{0,t}^i: t \in [0,T]\}, {\cal G}_T^{x^{i,u}} \tri \{ {\cal G}_{0,t}^{x^{i,u}}: t \in [0,T]\}, i=1, \ldots, {\mathbb Z}_N$,
and define all such adapted nonanticipative functions by
\begin{align}
\overline{\mathbb U}_{rel}^{i}[0, T] \tri \Big\{  u \in  L_{{\cal F}_{
T}^{ i}}^\infty([0,T],{\cal M}_1( {\mathbb A}^{i}))    : \:   u^i   \in   L_{{\cal G}_T^{ x^{i,u}}}^\infty([0,T],{\cal M}_1( {\mathbb A}^{i}))   \Big\}, \: \forall  i \in {\mathbb Z}_N.\label{naa1}
     \end{align}

\noi Next, we introduce the following additional assumptions.

\begin{assumptions}
\label{a-nf1}
The following holds.
\begin{description}
\item[(E2)] ${\mathbb U}_{rel}^{x^{i,u}}[0, T]$ is dense in  $\overline{\mathbb U}_{rel}^{i}[0, T], \forall i \in {\mathbb Z}_N$.

\end{description}
\end{assumptions}

\noi Under the additional Assumptions~\ref{a-nf1} we can prove  the following result.

\begin{theorem}
\label{thm-nf1}
Consider Problem~\ref{problem1}  with control strategies from ${\mathbb U}_{rel}^{(N),x^{u}}[0,T].$ Under Assumptions~\ref{A1-A4}, \ref{assumptionscost},  \ref{a-nf1}, and $|\varphi_x(x)+ \ell_x(t,x,u)|_{{\mathbb R}^n} \leq K ( 1 +| x|_{{\mathbb R}^n})$ we have,  \\
\bea
\inf_{ u \in  \times_{i=1}^N \overline{\mathbb U}_{rel}^{i}[0,T]} J(u)=  \inf_{ u \in \times_{i=1}^N {\mathbb U}_{rel}^{z^{i,u}}[0,T]} J(u).\label{equic}
\eea
\end{theorem}

\begin{proof} The assertion is obvious because of the density assumption   { (E2)} and the continuity of $J$ in the vague topology.

%  is assumed, it is sufficient to show that  as $u^{i, \alpha} \buildrel v \over\longrightarrow u^{i}$  in $\overline{\mathbb U}_{rel}^i[0,T]$, $\forall i \in {\mathbb Z}_N$,  then $J(u^\alpha) \longrightarrow J(u)$. This is done by following the derivation of Lemma~\ref{lemma3.1}, to show that ${\mathbb E} \sup_{s \in [0,t] } |x^\alpha(s)-x(s)|_{{\mathbb R}^n}$ converges to zero, and then use the mean value theorem   to show that $|J(u^\alpha)-J(u)|$ converges to zero, as $\alpha \longrightarrow 0$.
\end{proof}
\noi The point to be made regarding Theorem~\ref{thm-nf1} is that if $u \in {\mathbb U}_{rel}^{(N),x^{u}}[0,T]$ achieves the infimum of $J(u)$ then it is also optimal with respect to  $\overline{\mathbb U}_{rel}^{(N)}[0,T] \tri \times_{i=1}^N \overline{\mathbb U}_{rel}^{i}[0,T] $. Consequently, the  necessary conditions for feedback information structures $u \in {\mathbb U}_{rel}^{(N),x^{u}}[0,T]$ to be optimal are those for which nonanticipative information structures $u \in \overline{\mathbb U}_{rel}^{(N)}[0,T]$ are optimal.
%Consequently, we avoid including derivatives of $u$ with respect to $x$ in the variational equation.
\\

\noi In the next remark  we give an example for which Assumptions~\ref{a-nf1} hold, and hence Theorem~\ref{thm-nf1} is valid.

\begin{remark}
\label{ss}
Suppose $x^1$ and $x^2$ are governed by the following stochastic differential equations
\begin{align}
dx^1(t)=& f^1(t,x^1(t),u^1(t))dt + \sigma^1(t,x^1(t))dW^1(t), \hst x^1(0)=x_0^1, \label{ss1} \\
dx^2(t)=& f^2(t,x^1(t),x^2(t),u^1(t),u^2(t))dt + \sigma^2(t,x^1(t),x^2(t))dW^2(t), \hst x^2(0)=x_0^2, \label{ss2} \\
z^1(t)=&h^1(t,x^1(t)), \hst z^2(t)=h^2(t,x^1(t),x^2(t)), \hst t \in [0,T], \label{ss3}
\end{align}
where $h^1, h^2$ are measurable, $W^1(\cdot), W^2(\cdot)$ are independent, and $u^1 \in {\mathbb U}_{rel}^{1,z^{1,u^1}}[0,T], u^2 \in {\mathbb U}_{rel}^{2,z^{2,u^2}}[0,T]$. If we further assume that $\{\sigma^i(\cdot, \cdot)\}$ and their inverses  are bounded, then we can find $\overline{\mathbb U}_{rel}^i[0,T], i=1,2$ for which { (E2)} holds, and thus  Theorem~\ref{thm-nf1} holds. The structure of the stochastic dynamics (\ref{ss1}), (\ref{ss2}) can be generalized to more than two coupled systems.
\end{remark}

\noi Next, we introduce the following alternative theorem to Theorem~\ref{thm-nf1}, which does not employ Assumptions~\ref{a-nf1}.

\begin{theorem}
\label{geeq}
Consider Problem~\ref{problem1}  with strategies from ${\mathbb U}_{rel}^{(N),z^{u}}[0,T]$, under Assumptions~\ref{A1-A4}, \ref{assumptionscost}, \ref{existence}, \ref{a-nf} and $|\varphi_x(x)+ \ell_x(t,x,u)|_{{\mathbb R}^n} \leq K ( 1 +| x|_{{\mathbb R}^n})$. \\
Then  ${\mathbb U}_{rel}^{z^{i,u}}[0, T]$ is dense in  $\overline{\mathbb U}_{rel}^{i}[0, T], \forall i \in {\mathbb Z}_N$ and
\bea
\inf_{ u \in  \times_{i=1}^N \overline{\mathbb U}_{rel}^{i}[0,T]} J(u)=  \inf_{ u  \in \times_{i=1}^N  {\mathbb U}_{rel}^{z^{i,u}}[0,T]} J(u). \label{eqcost}
\eea
\end{theorem}

\begin{proof} The derivation is based on  \cite{bensoussan1981} but extended to relaxed strategies.   By Theorem~\ref{thm-nf}, for any $u^i \in {\mathbb U}_{rel}^{z^{i,u}}[0,T]$ which is ${\cal G}_T^{z^{i,u}}-$adapted we can define the set $\overline{\mathbb U}_{rel}^i [0,T], i=1, \ldots, N$ via (\ref{naa1}). For any  $u \in \overline{\mathbb U}_{rel}^{(N)}[0,T] \tri \times_{i=1}^N \overline{\mathbb U}_{rel}^i[0,T],$  $k =\frac{T}{M}$, and   any  test function  $\phi \in  C({\mathbb A}^{(N)}),$    define
\bea
u_{k,t}[\phi] \tri \left\{ \begin{array}{cccc} \int_{{\mathbb A}^{(N)}} \phi(\xi) u_0(d\xi) & \mbox{for} & 0 \leq t <k & u_0 \in {\mathbb A}^{(N)} \\
\frac{1}{k} \int_{(n-1)k}^{nk}   \int_{{\mathbb A}^{(N)}}\phi(\xi) u_s(d\xi) ds & \mbox{for} & nk \leq t < (n+1)k, & n=1,\ldots, M-1. \hso  \end{array} \right.
\eea
Clearly  $u_{k} \in \overline{\mathbb U}_{rel}^{(N)}[0,T] $, and $u_{k} \longrightarrow u$ in $
L_{{\cal F}_T}^\infty([0,T], {\cal M}_1({\mathbb A}^{(N)} ))$ in the weak star sense. We need to show that $u_{k} \in {\mathbb U}^{(N),z^{u_k}}[0,T]$. Let $x_k$ denote the trajectory corresponding to $u_k$, and ${\cal F}_{0,t}^{x_k^u}$ the $\sigma-$algebra generated by $\{x_k(s): 0 \leq s \leq t\}$. Define
\begin{align}
I_k(t) \tri \int_{0}^t \sigma(s,x_k(s))  dW(t)  
= x_k(t) -x(0) -\int_{0}^t f(s, x_k(s), u_k(s)) ds, \label{dc1}
\end{align}
and
\bea
W(t) = \int_{0}^t \sigma(s,x_k(s))^{-1} d I_k(s). \label{dc2}
\eea
Since $u_k \in \overline{\mathbb U}_{rel}^{(N)}[0,T] $, the process  $I_k(t)$ is ${\cal F}_{0,t}^{x_k^u}-$measurable, for $0 \leq t <k$. Hence,
\bea
{\cal F}_{0,t}^{x(0),W}= {\cal F}_{0,t}^{x_k^u}, \hst 0\leq t \leq k. \label{dc3}
\eea
Therefore, $u_{k,t}$ is ${\cal F}_{0,t}^{x_k^u}-$ measurable for $k \leq t \leq 2k$.  From the above equations it follows that (\ref{dc3}) also holds for $k \leq t \leq 2k$, and by induction that ${\cal F}_{0,t}^{x(0),W} = {\cal F}_{0,t}^{x_k^u}, \forall t \in [0,T]$. Therefore,  $u_{k,t}^i $ is also (weak star)  measurable with respect to ${\cal F}_{0,t}^{x_k^u}$. Hence , for any $u_t^i$  which is (weak star)  measurable with respect to a nonanticipative functional $z^i=h^i(t,x)$ there exists a nonanticipative functional of $\{x(0), W\}$ which realizes it.  By Theorem~\ref{thm-nf1} the derivation is complete. \end{proof}.

\begin{comment}
********{\bf Hi Bambos: Here there is a  mathematical problem. The relaxed controls are not Lebesgue measurable, they are only weak  star measurable. So it is not Lebesgue integrable as you imply in the above proof. See the construction given by the expression  (29). Strictly speaking one should not  even write like this  $ L_{{\cal F}_T} ^{\infty}(I,{\cal M}_1(A^{(N)}).$ This symbol means strong measurability and essential boundedness.  One way to resolve this is to use }
\bea
u_{k,t}(\varphi) = \left\{ \begin{array}{cccc} u_0(\varphi) & \mbox{for} & 0 \leq t <k & u_0 \in {\mathbb A}^{(N)} \\
\frac{1}{k} \int_{(n-1)k}^{nk} u_s(\varphi ) ds & \mbox{for} & nk \leq t \leq  (n+1)k, & n=1,\ldots, M-1, \forall \varphi \in BC(A^{N}) \end{array} \right.
\eea ********

{\bf
******** Hi Nasir: This has escaped from me. Thanks. Please check it again}

\end{comment}

Before we prove the optimality conditions we define the Hamiltonian system of equations.\\
\noi The  Hamiltonian is a real valued function
\bes
 {\mathbb  H}: [0, T] \times {\mathbb R}^n\times {\mathbb R}^n\times {\cal L}({\mathbb R}^m,{\mathbb R}^n)\times   {\cal  M}_1({\mathbb A}^{(N)}) \longrightarrow {\mathbb R}
\ees
  given  by
   \begin{align}
    {\mathbb H} (t,\xi,\zeta,M,\nu) \tri    \langle f(t,\xi,\nu),\zeta \rangle + tr (M^*\sigma(t,\xi,\nu))
     + \ell(t,\xi,\nu),  \hst  t \in  [0, T]. \label{h1}
    \end{align}
    \noi For any $u \in {\mathbb U}_{rel}^{(N)}[0,T]$, the adjoint process is $(\psi,Q) \in  L_{{\mathbb F}_T}^2([0,T], {\mathbb R}^n)\times L_{{\mathbb F}_T}^2([0,T] ,{\cal L}({\mathbb R}^m,{\mathbb R}^n))$ satisfies the following backward stochastic differential equation
\begin{align}
d\psi (t)  &= -f_x^{*}(t,x(t),u_t)\psi (t)  dt - V_{Q}(t) dt -\ell_x(t,x(t),u_t) dt + Q(t) dW(t), \hst t \in [0,T),    \nonumber    \\
&=- {\mathbb H}_x (t,x(t),\psi(t),Q(t),u_t) dt + Q(t)dW(t),      \label{adj1a} \\
 \psi(T) &=   \varphi_x(x(T))  \label{eq18}
 \end{align}
  where  $V_{Q} \in L_{{\mathbb F}_T}^2([0,T],{\mathbb R}^n)$ is   given by  $\langle V_{Q}(t),\zeta\rangle = tr (Q^*(t)\sigma_x(t,x(t),u_t; \zeta)), t \in [0,T]$ (e.g., $V_{Q}(t) = \sum_{k=1}^m \Big(\sigma_x^{(k)}(t,x(t),u_t)\Big)^*Q^{(k)}(t), \hst t \in [0,T],$ $\sigma^{(k)}$ is the $kth$ column of $\sigma$, $\sigma_x^{(k)}$ is the  derivative of $\sigma^{(k)}$ with respect to the state, for $k=1, 2, \ldots, m$, $Q^{(k)}$ is the $kth$ column of $Q$).\\
 In terms of the Hamiltonian, the state process satisfies the stochastic differential equation
\begin{align}
dx(t) &=f(t,x(t),u_t)dt + \sigma(t,x(t),u_t)dW(t), \hst t \in (0, T],  \nonumber  \\
        & = {\mathbb H}_\psi (t,x(t),\psi(t),Q(t),u_t)     dt + \sigma(t,x(t),u_t) dW(t), \label{st1a} \\
        x(0) &=  x_0 \label{st1i}
 \end{align}

\subsection{Necessary Conditions of Optimality}
\label{necessary}
In this section we state and prove the necessary conditions for team optimality. Specifically, given that $u^o \in {\mathbb U}_{rel}^{(N)}[0,T]$ or $u^o \in {\mathbb U}_{rel}^{(N),z^u}[0,T]$ is team optimal, we show that it leads naturally  to   the Hamiltonian system of equations (called necessary conditions).
The derivation is based on the semi martingale representation as in \cite{ahmed-charalambous2012a} with some modifications necessary to admit   decentralized strategies adapted to an arbitrary filtration.  \\

\noi In the following theorem we present  the necessary conditions of optimality for Problem~\ref{problem1}.

 \begin{theorem} (Necessary conditions for team optimality)
 \label{theorem5.1}
Consider Problem~\ref{problem1} under Assumptions~\ref{assumptionscost}, \ref{NC1}.

\begin{description}

\item[(I)] Suppose ${\mathbb F}_T = \sigma\{ x(0),W(t), t\in [0,T]\}$ and ${\mathbb U}_{rel}^{(N)}[0,T]$ is the class of relaxed controls adapted to this  filtration. For  an element $ u^o \in {\mathbb U}_{rel}^{(N)}[0,T]$ with the corresponding solution $x^o \in B_{{\mathbb F}_T}^{\infty}([0,T], L^2(\Omega,{\mathbb R}^n))$ to be team optimal, it is necessary  that
the following conditions  hold.
   \begin{description}

\item[(1)]  There exists a semi martingale  $m^o \in {\cal SM}_0^2[0,T]$ with the intensity process $({\psi}^o,Q^o) \in  L_{{\mathbb F}_T}^2([0,T],{\mathbb R}^n)\times L_{{\mathbb F}_T}^2([0,T],{\cal L}({\mathbb R}^m,{\mathbb R}^n))$.

 \item[(2) ]    The processes $\{u^o,x^o,\psi^o,Q^o\}$ satisfy the  inequality : \begin{align}
  \sum_{i=1}^N {\mathbb E} \int_{0}^{T}   {\mathbb H}(t,x^o(t)\psi^o(t), Q^o(t), u_t^{-i,o},u_t^i-u_t^{i,o})dt
  \geq 0, \hst \forall u \in {\mathbb U}_{rel}^{(N)}[0,T]. \label{eq16}
\end{align}

\item[(3)]  The process $({\psi}^o,Q^o)$  is the  unique solution of the backward stochastic differential equation (\ref{adj1a}), (\ref{eq18}) and  that, for  ${\cal G}_{0,t}^i   \subset {\mathbb F}_{0,t},$ the control  $u^o \in {\mathbb U}_{rel}^{(N)}[0,T]$ satisfies  the  point wise almost sure inequalities.
\begin{align}
& {\mathbb E} \Big\{  {\mathbb H}(t,x^o(t),\psi^0(t),Q^o(t),u_t^{-i,o},\nu^i )|{\cal G}_{0, t}^i \Big\}    \geq   {\mathbb E} \Big\{ {\mathbb H}(t,x^o(t),\psi^o(t),Q^o(t),u_t^{o})|{\cal G}_{0, t}^i \Big\} ,  \nonumber \\
&  \forall \nu^i \in {\cal M}_1({\mathbb A}^i),  a.e. t \in [0,T], {\mathbb P}|_{{\cal G}_{0,t}^i}- a.s., i=1,2,\ldots, N.   \label{eqh35}
\end{align}
\end{description}
\item[(II)] Suppose ${\mathbb F}_T$ is as above,  and  the Assumption~\ref{a-nf1} holds. \noi For  an element $ u^o \in {\mathbb U}_{rel}^{(N),z^u}[0,T]$  with the corresponding solution $x^o \in B_{{\mathbb F}_T}^{\infty}([0,T], L^2(\Omega,{\mathbb R}^n))$ to be team optimal, it is necessary  that
the statements of Part {\bf (I)} hold with ${\cal G}_{0,t}^i$ replaced by ${\cal G}_{0,t}^{z^{i,u}}, \forall t \in [0,T]$.

%\item[(III)] Suppose  $\{W(t): t \in [0,T]\}$ is an ${\mathbb F}_T-$martingale, and this filtration is right continuous (not necessarily generated by $x(0)$ and $\{W(t): t \in [0,T]\})$  and  the additional Assumption~\ref{a-nf1} holds. Then, for  an element $ u^o \in {\mathbb U}_{rel}^{(N),z^u}[0,T]$  with the corresponding solution $x^o \in B_{{\mathbb F}_T}^{\infty}([0,T], L^2(\Omega,{\mathbb R}^n))$ to be team optimal, it is necessary that the statements of Part {\bf (I)} hold  with  (\ref{adj1a}), (\ref{eq18})  replaced by
%\begin{align}
%d\psi (t)  =&- {\mathbb H}_x (t,x(t),\psi(t),Q(t),u_t) dt + Q(t)dW(t) + dM(t),  \hst t \in [0,T) \label{adjg1} \\
% \psi(T)=&   \varphi_x(x(T)),  \label{adjg2}
% \end{align}
%where $\{M(t): t \in [0,T]\}$ is a square integrable right continuous ${\mathbb F}_T $ martingale, orthogonal to $\{W(t): t \in [0,T]\}$, and ${\cal G}_{0,t}^i$ replaced by ${\cal G}_{0,t}^{z^{i,u}}, \forall t \in [0,T], \forall i \in {\mathbb Z}_N$.
 \end{description}

\end{theorem}

 \begin{proof}  The derivation of {\bf (1), (2)} follows closely the basic steps  of centralized strategies in \cite{ahmed-charalambous2012a}, from which the derivation of team necessary conditions of optimality  {\bf (3)} are established.\\
   {\bf (I)}. {\bf (1)} Suppose $u^o \in {\mathbb U}_{rel}^{(N)}[0,T]$ is an optimal team decision  and $u \in {\mathbb U}_{rel}^{(N)}[0,T]$ any other admissible decision.  Since ${\mathbb U}_{rel}^i[0,T]$ is convex $\forall i \in {\mathbb Z}_N$, we have, for any $\varepsilon \in [0,1]$, $ u_t^{i,\varepsilon} \tri u_t^{i,o} + \varepsilon (u_t^i-u_t^{i,o}) \in {\mathbb U}_{rel}^i[0,T], \forall i \in {\mathbb Z}_N.$ Let $x^\varepsilon (\cdot) \equiv x^{\varepsilon}(\cdot;u^\varepsilon(\cdot)), x^{o}(\cdot)\equiv x^o(\cdot;u^o(\cdot)) \in B_{{\mathbb F}_T}^{\infty}([0,T],L^2(\Omega,{\mathbb R}^n)) $ denote the  solutions of the system  (\ref{sd1}) and  (\ref{st1i})  corresponding to $u^{\varepsilon}(\cdot)$ and $u^o(\cdot)$, respectively. Since $u^o(\cdot) \in {\mathbb U}_{rel}^{(N)}[0,T]$ is optimal it is clear that
 \begin{eqnarray}
   J(u^{\varepsilon})-J(u^o) \geq 0, \hst   \forall \varepsilon \in [0,1], \hso \forall u \in {\mathbb U}_{rel}^{(N)}[0,T].  \label{eq19}
    \end{eqnarray}

\noi Define the  Gateaux differential  of $J$ at $u^o$ in the direction $u-u^o$ by
\bes
dJ(u^o,u-u^0) \tri \lim_{ \varepsilon \downarrow 0}   \frac{ J(u^{\varepsilon})-J(u^o)}{\varepsilon} \equiv \frac{d}{d \varepsilon} J(u^\varepsilon)|_{\varepsilon =0}.
\ees
\noi Dividing the expression (\ref{eq19}) by $\varepsilon$ and letting $\varepsilon \downarrow 0$ we obtain
 \begin{eqnarray}
 dJ(u^o,u-u^0) = L(Z) +  \sum_{i=1}^N {\mathbb E} \int_0^T \ell(t,x^o(t),u^{-i,o},u_t^i-u_t^{i,o}) dt \geq 0, \hst \forall u \in {\mathbb U}_{rel}^{(N)}[0,T], \label{eq20}
 \end{eqnarray}
where $L(Z)$ is given by the functional
\begin{eqnarray}
L(Z) = {\mathbb E}\biggl\{  \int_{0}^{T} \la\ell_x(t,x^o(t),u_t^o),Z(t)\ra~ dt + \la\varphi_x(x^o(T)),Z(T)\ra\biggr\}. \label{eq21}
\end{eqnarray}

\noi  Since by Lemma~\ref{lemma4.1}, the process $Z(\cdot)\in B_{{\mathbb F}_T}^{\infty}([0,T],L^2(\Omega,{\mathbb R}^n))$ and it is also   continuous ${\mathbb P}-$a.s   it follows from  Assumptions~\ref{assumptionscost}, {\bf (B2)}, and Assumptions~\ref{NC1}, that $Z \longrightarrow L(Z)$ is a continuous linear functional. Further, by Lemma~\ref{lemma4.1}, $\eta \longrightarrow Z$ is a continuous linear map from the Hilbert space ${\cal SM}_0^2[0,T]$ to the B-space $B_{{\mathbb F}_T}^{\infty}([0,T],L^2(\Omega,{\mathbb R}^n)) $  given by the expression (\ref{eq10}). Thus the composition map  $\eta \longrightarrow Z \longrightarrow L(Z) \equiv \tilde L (\eta)$ is a continuous linear functional on ${\cal SM}_0^2[0,T].$ Then by virtue of  Riesz representation theorem for Hilbert spaces, there exists a semi martingale $m^o \in {\cal SM}_0^2[0,T]$ with intensity $(\psi^o,Q^o) \in  L_{{\mathbb F}_T}^2([0,T],{\mathbb R}^n)\times L_{{\mathbb F}_T}^2([0,T],{\cal L}({\mathbb R}^m,{\mathbb R}^n))$ such that

\begin{align} L(Z) \tri \tilde L (\eta) = (m^o,\eta)_{{\cal SM}_0^2[0,T]} = & \sum_{i=1}^N {\mathbb E} \int_{0}^{T} \la\psi^o(t), f(t,x^o(t),u^{-i,o},u_t^i-u_t^{i,o})\ra dt \nonumber \\
&+\sum_{i=1}^N {\mathbb E} \int_{0}^{T} tr (Q^{o,*}(t) \sigma(t,x^o(t),u^{-i,o},u_t^i-u_t^{i,o})) dt . \label{eq22}
 \end{align}
 This proves  {\bf (1)}.\\

\noi {\bf (2)} Substituting (\ref{eq22}) into (\ref{eq20}) we obtain the following variational equation.
 \begin{align}
 dJ(u^o,u-u^0) =  &  \sum_{i=1}^N {\mathbb E} \int_{0}^{T} \la\psi^o(t), f(t,x^o(t),u^{-i,o},u_t^i-u_t^{i,o})\ra dt   \nonumber \\
 &+\sum_{i=1}^N {\mathbb E} \int_{0}^{T} tr (Q^{o,*}(t) \sigma(t,x^o(t),u^{-i,o},u_t^i-u_t^{i,o})) dt
  \nonumber \\
 & + \sum_{i=1}^N {\mathbb E} \int_0^T \ell(t,x^o(t),u^{-i,o},u_t^i-u_t^{i,o}) dt
  \geq 0, \hst \forall u \in {\mathbb U}_{rel}^{(N)}[0,T]. \label{eq20n}
 \end{align}
 It follows from the definition of the Hamiltonian that the  inequality (\ref{eq20n})  is precisely (\ref{eq16}) along with the pair $\{ (\psi^o(t), Q^o(t)): t \in [0,T]\}$. This completes the proof of {\bf (2)}. \\

 \noi {\bf (3)} Next, we prove that the pair $\{ (\psi^{o}(t),Q^{o}(t)): t \in [0, T]\})$ is given by the solution of the  adjoint equations (\ref{adj1a}),  (\ref{eq18}). Computing  the It\^o differential   of the scalar product $\la Z,\psi^o\ra$ and  integrating this over $[0,T]$, it follows from the variational equation (\ref{eq9})  that

\begin{align}
 {\mathbb E} \la Z(T),\psi^{o}(T) \ra = &    {\mathbb  E} \biggl\{ \int_{0}^{T} \la Z(t), f_x^{*}(t,x^o(t),u_t^{o}) \psi^{o}(t) dt
  + \sigma_x^{*}(t,x^o,u_t^o;\psi^{o})dW(t) + d\psi^{o}(t) \ra  \nonumber \\
 &+ \sum_{i=1}^N \int_{0}^{T} \la f(t,x^o(t),u_t^{-i,o},u_t^i-u_t^{i,o}),\psi^{o}(t)\ra dt  \nonumber \\
 & +\sum_{i=1}^N \int_{0}^{T} \la\sigma^{*}(t,x^o(t),u^{-i,o},u_t^{i}-u_t^{i,o})\psi^{o}(t),dW(t)\ra +  \int_{0}^{T} <dZ,d\psi^{o}>(t) \biggr\}  \label{eq25n} \\
& =    {\mathbb  E} \biggl\{ \int_{0}^{T} \la Z(t), f_x^{*}(t,x^o(t),u_t^{o}) \psi^{o}(t) dt
  + d\psi^{o}(t) \ra  \nonumber \\
 &+ \sum_{i=1}^N \int_{0}^{T} \la f(t,x^o(t),u^{-i,o},u_t^i-u_t^{i,o}),\psi^{o}(t)\ra dt  \nonumber \\
 &  \int_{0}^{T} <dZ,d\psi^{o}>(t) \biggr\} ,  \label{eq26n}
 \end{align}
 where the last bracket  $<\cdot,\cdot >$ in each of the above  expressions is the quadratic variation between the two processes, and  the  stochastic integrals in (\ref{eq25n}) have zero expectation giving (\ref{eq26n}).
Since  It\^o derivatives of the variation process $\{Z(t): t \in [0, T]\}$ and the adjoint process $\{\psi^{o}(t): t \in [0,T]\}$ have the form
 \begin{align}  dZ(t) &= \hbox{bounded variation terms} + \sigma_{x}(t,x^o(t),u_t^{o}; Z(t)) dW(t)  \nonumber \\
 &+ \sum_{i=1}^N\sigma(t,x^o(t),u_t^{-i,o},u_t^i-u_t^{i,o}) dW(t), \hso Z(0)=0,  \hst t \in (0,T],  \label{gv1} \\
  d\psi^{o}(t) &= \hbox{bounded variation terms}  + Q^{o}(t) dW(t), \hst \psi^{o}(T)= \varphi_x(x^o(T)) , \label{ga1}
  \end{align}
 their quadratic variation  is given by
\begin{align}
  {\mathbb  E} \int_{0}^{T} <dZ,d\psi^{o}>(t) =  &   {\mathbb E} \Big\{ \int_{0}^{T}  tr (Q^{o,*}\sigma_x(t,x^o(t),u_t^{o};Z(t)) ) dt \Big\} \nonumber \\
  &+\sum_{i=1}^N    {\mathbb E} \Big\{ \int_{0}^{T}   tr(Q^{o,*}(t)\sigma(t,x^o(t), u_t^{-i,o},u_t^i-u_t^{i,o}))dt \Big\}. \label{eq27n}
  \end{align}
   The first term on the right hand side of the above expression is linear in $Z,$ hence  there exists a process  $
   \{V_{Q^{o}}(t): t \in[0,T]\},$  given by the following expression
    \begin{eqnarray}
    \la V_{Q^{o}}(t),Z(t) \ra \tri  tr ( Q^{o,*}(t) \sigma_x(t,x^o(t),u_t^{o};Z(t)) ). \label{eq28n}
     \end{eqnarray}
     By Assumptions~\ref{NC1}, $\sigma$ has uniformly bounded  spatial first derivative and it follows from the semi martingale representation  that $Q^{o} \in L_{{\mathbb F}_T}^2([0, T],{\cal L}({\mathbb R}^m,{\mathbb R}^n))$ and hence $V_{Q^{o}} \in L_{{\mathbb F}_T}^2([0,T],{\mathbb R}^n).$ Substituting (\ref{eq28n}) into (\ref{eq27n}) and   (\ref{eq27n}) into (\ref{eq26n}), we obtain
\begin{align}
{\mathbb  E} (Z(T),\psi^{o}(T)) =  {\mathbb E} &\biggl\{ \int_{0}^{T} \la Z(t), f_x^{*}(t,x^o(t),u_t^{o})\psi^{o} dt + V_{Q^{o}}(t) dt - Q^{o}(t)dW(t) +  d\psi^{o}(t) \ra \biggr\} \nonumber \\
& + \sum_{i=1}^N {\mathbb E}  \biggl\{ \int_{0}^{T} \la f(t,x^o(t),u_t^{-i,o}, u_t^i-u_t^{i,o}),\psi^{o}(t) \ra dt  \nonumber \\
& + tr(Q^{o,*}(t)\sigma(t,x^o(t),u_t^{-i,o},u_t^i-u_t^{i,o})) dt \biggr\}.  \label{eq29n}
\end{align}
Thus,  by setting
\begin{align}
 d\psi^{o}(t) & =- f_x^{*}(t,x^o(t),u_t^{o})\psi^{o}(t) dt -V_{Q^{o}}(t) dt  +Q^{o}(t) dW(t)
  -\ell_x(t,x^o(t),u_t^{o}) dt, \hst t \in [0,T) \label{eq30an} \\
 \psi^{o}(T)& = \varphi_x(x^o(T)),  \label{eq30n}
\end{align}
 it follows from (\ref{eq29n}) and the expression for the functional $L(\cdot)$ given by (\ref{eq21})
that
\begin{align}
 L(Z) &= {\mathbb E}  \Big\{\la Z(T),\psi^{o}(T)\ra + \int_{0}^{T} \la Z(t),\ell_x(t,x^o(t),u_t^{o})\ra dt \Big\}\nonumber \\
 &=  \sum_{i=1}^N {\mathbb E}  \Big\{ \int_{0}^{T} \la f(t,x^o(t),u_t^{-i,o},u_t^{i}-u_t^{i,o}),\psi^{o}(t)\ra
 + tr(Q^{o,*}\sigma(t,x^o(t),u_t^{-i,o},u_t^i-u_t^{i,o})) dt \Big\} .   \label{eq31n}
\end{align}
Substituting (\ref{eq31n}) into (\ref{eq20}) we again obtain (\ref{eq16}), as expected. This is precisely what was obtained by the semi martingale argument giving (\ref{eq22}). Thus  the pair $\{(\psi^{o}(t),Q^{o}(t)): t \in [0,T]\}$ must satisfy the backward stochastic differential equation (\ref{eq30an}), (\ref{eq30n}), which is precisely the adjoint equation given by (\ref{adj1a}), (\ref{eq18}). Since $\psi^{o}$ satisfies the stochastic differential equation and $T$ is finite,  it follows from the classical  theory of It\^o differential equations that $\psi^{o}$  is actually   an element of $B_{{\mathbb F}_T}^{\infty}([0,T],L^2(\Omega,{\mathbb R}^n)) \subset L_{{\mathbb F}_T}^2([0,T],{\mathbb R}^n).$ In other words, $\psi^{o}$ is more regular than predicted by semi martingale theory.  Hence, by our Assumptions on $\sigma$ it is easy to verify that   $\sigma_x^{*}(t,x^o(t),u_t^{o};\psi^{o}(t))  \in L_{{\mathbb F}_T}^2([0,T],{\cal L}( {\mathbb R}^m, {\mathbb R}^n)) $   and $$\sigma^{*}(t,x^o(t),u_t^{-i,o},u_t^{i}-u_t^{i,o})\psi^{o}(t) \in L_{{\mathbb F}_T}^2([0,T],{\mathbb R}^n), \hst i=1, \ldots, N.$$
This proves the first part of  {\bf (3)}.\\
Now we show (\ref{eqh35}). Write (\ref{eq16}) in  terms of the Hamiltonian as follows.

  \begin{align}
\sum_{i=1}^N  {\mathbb E} \Big\{ \int_0^T   {\mathbb H}(t,x^o(t),\psi^o(t),Q^o(t),u_t^{-i,o},u_t^i-u_t^{i,o}) dt \Big\}
  \geq 0,  \hst \forall u \in {\mathbb U}_{rel}^{(N)}[0,T], \label{eq32}
 \end{align}
 where the triple  $\{x^o,\psi^o,Q^o\}$ is  the unique solution of the Hamiltonian system (\ref{adj1a}), (\ref{eq18}), (\ref{st1a}), (\ref{st1i}).  By using the property of conditional expectation  then
\begin{align} \sum_{i=1}^N   {\mathbb E} \biggl\{   \int_{0}^{T}  {\mathbb  E} \Big\{  {\mathbb H}(t,x^o(t),\psi^o(t),Q^o(t),u_t^{-i,o},u_t^i-u_t^{i,o})|{\cal G}_{0, t}^i\Big\}  dt \biggr\}  \geq 0, \hso \forall u \in {\mathbb U}_{rel}^{(N)}[0,T].  \label{eq36a}
\end{align}
\noi Let  $t \in (0,T),$  $\omega \in \Omega$ and $\varepsilon >0$, and  consider the sets $I_{\varepsilon}^i \equiv [t,t+\varepsilon] \subset [0,T]$ and  $\Omega_{\varepsilon}^i (\subset \Omega) \in {\cal G}_{0,t}^i$  containing $\omega$ such that $|I_{\varepsilon}^i| \rightarrow 0$  and ${\mathbb P}(\Omega_{\varepsilon}^i) \rightarrow 0$ as $\varepsilon \rightarrow 0,$ for $i=1,2, \ldots, N$. For any sub-sigma algebra ${\cal G} \subset {\mathbb F}$, let ${\mathbb P}|_{{\cal G}}$ denote the restriction of the probability measure ${\mathbb P}$ on to the $\sigma$-algebra ${\cal G}.$   For any (vaguely) ${\cal G}_{0, t}^i-$adapted  $\nu^i \in {\cal M}_1({\mathbb A}^i),$ construct
\bea
 u_t^i = \begin{cases}  \nu^i & ~ \mbox{for}~~ (t,\omega) \in I_{\varepsilon}^i \times \Omega_{\varepsilon}^i  \\   u_t^{i,o} & \mbox{ otherwise}           \end{cases} \hst i=1,2, \ldots, N. \label{cc1}
 \eea
 Clearly, it follows from the above construction that $u^i \in {\mathbb U}_{rel}^i[0,T].$  Substituting  (\ref{cc1})  in (\ref{eq36a}) we obtain the following inequality
\begin{align}
\sum_{i=1}^N \int_{\Omega_{\varepsilon}^i\times I_{\varepsilon}^i}  {\mathbb E} \Big\{  {\mathbb H}(t,x^o(t),\psi^o(t), &Q^o(t),u_t^{-i,o},\nu^i-u_t^{i,o})|{\cal G}_{0,t}^i\Big\}dt
   \geq  0, \nonumber \\
&    \forall \nu^i \in {\cal M}_1({\mathbb A}^i),  a.e. t \in [0,T], {\mathbb P}|_{{\cal G}_{0,t}^i}- a.s., \hso i=1,2,\ldots,N.\label{eq37}
   \end{align}
   Letting $|I_{\varepsilon}^i|$ denote the Lebesgue  measure of the set $I_{\varepsilon}^i$ and dividing the above expression  by the product measure ${\mathbb P}(\Omega_{\varepsilon}^i)|I_{\varepsilon}^i|$ and letting $\varepsilon \rightarrow 0$ we arrive at the following inequality.
\begin{align}   \sum_{i=1}^N  {\mathbb E} & \Big\{  {\mathbb H}(t,x^o(t),\psi^o(t),Q^o(t),u_t^{-i,o},\nu^i)|{\cal G}_{0,t}^i\Big\}
 \geq \sum_{i=1}^N  {\mathbb E} \Big\{ {\mathbb H}(t,x^o(t),\psi^o(t),Q^o(t),u_t^{-i,o},u_t^{i,o})|{\cal G}_{0,t}^i\Big\}, \nonumber \\
 &  \forall \nu^i \in {\cal M}_1({\mathbb A}^i),  a.e. t \in [0,T], {\mathbb P}|_{{\cal G}_{0,t}^i}- a.s., i=1,2,\ldots,N. \label{eq37ccc}
\end{align}
To complete the proof  of {\bf (3)}  define
\begin{align}
g^i(t,\omega) \tri {\mathbb E} \Big\{  {\mathbb H} (t,x^o(t),\psi^o(t),Q^o(t),u_t^{-i,o},\nu^i-u_t^{i,o})|{\cal G}_{0,t}^i\Big\},  \hst  t \in [0,T], \hso \forall i \in {\mathbb Z}_N.  \label{eq37c}
   \end{align}
We shall show that
\bea
g^i(t,\omega) \geq 0,  \hso \forall \nu^i \in {\cal M}_1({\mathbb A}^i), \: a.e. \: t \in [0,T], \hso {\mathbb P}|_{{\cal G}_{0,t}^i}- a.s., \: \forall i \in {\mathbb Z}_N. \label{eq37ab}
\eea
Suppose for some $i \in {\mathbb Z}_N$, (\ref{eq37ab}) does  not hold, and let $A^i \tri \{(t,\omega): g^i(t,\omega)<0 \}$. Since $g^i(t) $ is $ {\cal G}_{0, t}^i-$measurable $\forall t \in [0,T]$   we can choose $u^i$  in  (\ref{eq37ccc}) as   $$u_t^i \tri \left\{ \begin{array}{l} \nu ~\mbox{on}~ A^i \\ u_t^{i,o} \: \mbox{outside} \: A^i \end{array} \right.$$
 together with $u_t^j =u_t^{j,o}, j \neq i, j \in {\mathbb Z}_N$. Substituting this in  (\ref{eq37ccc}) we arrive at   $\int_{A^i} g^i(t,\omega) ds  \: d{\mathbb P} \geq 0$, which contradicts the definition of $A^i$, unless $A^i$ has measure zero. Hence, (\ref{eq37ab}) holds which is precisely (\ref{eqh35}). This completes Part {\bf (I)}.\\
{\bf (II)}. By Theorem~\ref{thm-nf1} the necessary conditions for team optimality  satisfy those  in Part {\bf (I)} with ${\cal G}_{0,t}^i$ replaced by ${\cal G}_{0,t}^{z^{i,u}}$. \\

%{\bf (III)}. By utilizing Remark~\ref{martgeneral}  we verify by repeating the derivation of Part {\bf (I)}, that only the  adjoint (\ref{ga1}) is  modified to include in its right hand side the extra martingale term $dM(t)$. However, since $\{M(t): t \in [0,T]\}$ is orthogonal to the martingale terms with respect to the Brownian motion $\{W(t): t \in [0,T]\}$ nothing else in the derivation is affected, except the adjoint equation for $\{\psi(t): t \in [0,T]\}$ to the one claimed.
\end{proof}

The following remark helps identifying the martingale term in the adjoint process.

\begin{remark}
\label{mart}
 The arguments in the derivation of Theorem~\ref{theorem5.1} involving  the Riesz representation theorem for Hilbert space martingales,  determine the martingale term of the adjoint process $M_t = \int_{0}^t  \psi_x^o(s) \sigma(s,x^o(s),u_s^o)dW(s)$, dual to the first martingale term in the variational equation (\ref{eq9}), provided $\psi_x(\cdot)$ exists (i.e., $f_{xx}, \sigma_{xx}, \ell_{xx}, \varphi_{xx}$ exist and are uniformly bounded).   Hence, $Q$ in the adjoint equation (\ref{adj1a}), is identified as $ Q(t) \equiv \psi_x(t) \sigma(t,x(t),u_t)$. %Thus, we can view $\psi^o(t) =V_x(t,x^o(t))$, where $V(t,x)$ is the vlaue function,  and  write in feedback form $\psi_x^o(t,x)=V_{xx}(t,x)$ and $\psi_x^o(t) \equiv \psi_x^o(t,x^o(t))=V_{xx}(t,x^o(t))$.
When the diffusion term $\sigma(\cdot,\cdot,\cdot)$  is independent of $x$, given by $\sigma(t,u)$, then since $\la V_Q(t),\zeta \ra = tr(Q^*(t) \sigma_x(t,x,u_t;\zeta))$ we have $V_{Q}(t)=0, \forall t \in [0,T]$  (e,g., the spatial derivative of the diffusion term is zero).
\end{remark}

 It is interesting to note that the necessary conditions, for a $u^o \in {\mathbb U}_{rel}^{(N)}[0,T]$ or  $u^o \in {\mathbb U}_{rel}^{(N),z^u}[0,T]$ to be  a person-by-person optimal policy,  can be derived following similar  steps  as given in  Theorem~\ref{theorem5.1}, and that these necessary conditions are the same as the necessary conditions for the team optimal strategy.
This is   stated as a Corollary.
\begin{corollary} (Necessary conditions for person-by-person optimality)
 \label{corollary5.1}
   Consider Problem~\ref{problem2} under Assumptions~\ref{assumptionscost}, \ref{NC1}.
   Under the conditions of Theorem~\ref{theorem5.1}, Part  {\bf (I)}, for  an element $ u^o \in {\mathbb U}_{rel}^{(N)}[0,T]$ with the corresponding solution $x^o \in B_{{\mathbb F}_T}^{\infty}([0,T], L^2(\Omega,{\mathbb R}^n))$ to be a person-by-person optimal strategy, it is necessary  that
 statements {\bf (1), (3)} of Theorem~\ref{theorem5.1}, and Part {\bf I},  with  statement {\bf (2)} replaced by  %\item[(2')]
   %\begin{description}
  \begin{align}
  {\mathbb E} \int_{0}^{T} {\mathbb H}(t,x^o(t), \psi^o(t), Q^o(t), u_t^{-i,o},u_t^i-u_t^{i,o})dt
  \geq 0, \hst \forall u^i \in {\mathbb U}_{rel}^{i}[0,T], \hso \forall i \in {\mathbb Z}_N. \label{eq16c}
\end{align} hold.
Similar conclusions hold for strategies  ${\mathbb U}_{rel}^{(N),z^u}[0,T].$
%\end{description}
\end{corollary}

 \begin{proof}  Primarily, the derivation is based on the same procedure as that of Theorem~\ref{theorem5.1}. The only difference is,  that in this case,  the variations of the DM policies are carried out in the direction of individual members while the rest of the  members carry optimal policy.  \\

\end{proof}

Clearly, every team optimal strategy for Problem~\ref{problem1} is  a person-by-person optimal strategy for Problem~\ref{problem2}.  Hence   person-by-person optimality is  weaker than team optimality. By comparing the statements of Theorem~\ref{theorem5.1} and Corollary~\ref{corollary5.1}, it is clear that statements {\bf (1)} and {\bf (3)} coincide, while the only difference are  the variational inequalities  (\ref{eq16}) and  (\ref{eq16c}).  However, (\ref{eq16c}) implies (\ref{eq16}), and it can be shown that (\ref{eq16}) implies (\ref{eq16c}). Indeed, if (\ref{eq16c}) is violated for some $j \in {\mathbb Z}_N$ then by choosing all other $u^i=u^{i,o}, \forall i \in {\mathbb Z}_N, i\neq j,$  the right side of (\ref{eq16}) will be negative, which is a contradiction.   This observation is new, and has not been documented in the static team game literature \cite{krainak-speyer-marcus1982a}.

\begin{remark}
\label{remark5.2}
From the above necessary conditions one can deduce the necessary conditions for full centralized information   and partial centralized  information. We state these conditions below.

%\begin{description}
{\bf (1)  Centralized Full Information Structures.}  Consider   Problem~\ref{problem1}  under the conditions of Theorem~\ref{theorem5.1}, Part {\bf (I)}, and assume $u^i$ are adapted to ${\mathbb F}_T$, $\forall  i\in {\mathbb Z}_N$. The necessary conditions are given by  the following point wise almost sure inequalities
 \begin{eqnarray}
  {\mathbb H}(t,x^o(t),\psi^o(t),Q^o(t),\mu )    \geq {\mathbb  H}(t,x^o(t),\psi^o(t),Q^o(t),u_t^{o}), \nonumber \\
\:  \forall \mu \in {\cal M}_1({\mathbb A}^{(N)}),  \: a.e. \: t \in [0,T], \: {\mathbb P}-a.s.,   \label{eq35a}  \end{eqnarray}

%or equivalently,
%
% \begin{eqnarray}
%  {\mathbb H}(t,x^o(t),\psi^o(t),Q^o(t),u_t^{o})   = \min_{ \mu \in {\cal M}_1({\mathbb A}^{(N)})}  {\mathbb  H}(t,x^o(t),\psi^o(t),Q^o(t),\mu),  \\
%\:    a.e. \: t \in [0,T], \: {\mathbb P}-a.s.,   \label{eq35ab}
% \end{eqnarray}
where $\{x^o(t), \psi^o(t), Q^o(t): t \in [0,T]\}$ are the solutions of the Hamiltonian system (\ref{st1a}), (\ref{st1i}), (\ref{adj1a}), (\ref{eq18}). This corresponds to the classical case \cite{yong-zhou1999}. \\
Moreover, if the strategies are based on centralized  state feedback information, that is, $u^i$ are adapted to the information ${\cal G}_T^{x^u}, \forall i \in {\mathbb Z}_N$, then under the conditions of Theorem~\ref{theorem5.1}, Part {\bf (II)} the previous optimality conditions are replaced by
\begin{eqnarray}
 {\mathbb E} \Big\{  {\mathbb H}(t,x^o(t),\psi^o(t),Q^o(t),\mu )|{\cal G}_{0,t}^{x^o}\Big\} \geq   {\mathbb E}\Big\{  {\mathbb H}(t,x^o(t),\psi^o(t),Q^o(t),u_t^o)|{\cal G}_{0,t}^{x^o}\Big\}, \nonumber \\
 \forall \mu \in {\cal M}_1({\mathbb A}^{(N)}), a.e. t \in [0,T], {\mathbb P}|_{{\cal G}_{0,t}^{x^o}}-a.s. \label{eq35abc}
 \end{eqnarray}
%and under the conditions of Theorem~\ref{theorem5.1}, Part {\bf (III)} the adjoint process is given by (\ref{adjg1}), (\ref{adjg2}).\\

{\bf (2) Centralized Partial Information Structures.}  Consider Problem~\ref{problem1}  under the conditions of Theorem~\ref{theorem5.1}, Part {\bf (I)} and Part {\bf (II)} and suppose that each  $u^i$ is  adapted to   the centralized partial information ${\cal G}_T\subset {\mathbb F}_T$, and   ${\cal G}_T^{z^u} \subset {\cal F}_{0,T}^{x^u}$, respectively. Then the necessary condition is given by

\begin{eqnarray}
 {\mathbb E} \Big\{  {\mathbb H}(t,x^o(t),\psi^o(t),Q^o(t),\mu )|{\cal K}_{0,t}  \Big\} \geq   {\mathbb E}\Big\{  {\mathbb H}(t,x^o(t),\psi(t),Q(t),u_t^o)|{\cal K}_{0,t}\Big\}, \nonumber  \\
\forall \mu \in {\cal M}_1({\mathbb A}^{(N)}), a.e. t \in [0,T], {\mathbb P}|_{{\cal K}_{0,t}}-a.s.
 \end{eqnarray}
 where ${\cal K}_{0,t}$ is a sub-sigma algebra of any of the sigma algebras indicated above.

 %\end{description}
\end{remark}

Finally, we mention two important results derived in \cite{ahmed-charalambous2012a} which have direct extensions to the current paper. The first addresses  existence of measurable relaxed team optimal strategy associated with the minimization of the Hamiltonian, and the second addresses existence of realizable relaxed strategies by regular strategies.

\subsection{Sufficient Conditions of Optimality}
\label{sufficient}
In this section, we show that the necessary conditions of optimality (\ref{eqh35}) are  also sufficient under certain  convexity conditions. \\

\begin{theorem} (Sufficient conditions for team optimality)
 \label{theorem5.1s}
   Consider Problem~\ref{problem1} and suppose Assumptions~\ref{assumptionscost}, \ref{NC1} hold.
   Under the conditions of Theorem~\ref{theorem5.1}, Part {\bf (I)},  let $( u^o(\cdot), x^o(\cdot))$ denote any control-state pair  (decision-state)  and let $\psi^o(\cdot)$ the corresponding adjoint processes. \\
   Suppose the
 following conditions hold:

\begin{description}

\item[(C4)]  ${\mathbb H} (t, \cdot,\zeta,M,\nu),   t \in  [0, T]$ is convex in $\xi \in {\mathbb R}^n$;

 \item[(C5)] $\varphi(\cdot)$ is convex in $\xi \in {\mathbb R}^n$.
\end{description}
\noi Then $(u^o(\cdot),x^o(\cdot))$ is  team optimal  if it satisfies (\ref{eqh35}). In other words, necessary conditions are also sufficient.
For feedback strategies ${\mathbb U}_{rel}^{(N),z^u}[0,T]$ the same  statement holds under the conditions of Theorem~\ref{theorem5.1}, Part {\bf (II)}.
\end{theorem}

 \begin{proof}  We shall prove the sufficiency  under the conditions of Theorem~\ref{theorem5.1}, {\bf (I)}, that is, the admissible strategies ${\mathbb U}_{rel}^{(N)}[0,T]$, since the derivation is precisely the same for the case Part {\bf (II)}.
 Let $u^o \in {\mathbb U}_{rel}^{(N)}[0,T]$ denote a candidate for the optimal team decision and $u \in {\mathbb U}_{rel}^{(N)}[0,T]$
any other decision. Then
 \begin{align}
 J(u^o) -J(u)=
    {\mathbb E} \biggl\{   \int_{0}^{T}  \Big(\ell(t,x^o(t),u_t^{o})  -\ell(t,x(t),u_t) \Big)  dt
     + \Big(\varphi(x^o(T)) - \varphi(x(T))\Big)  \biggr\}    . \label{s1}
  \end{align}
By the convexity of $\varphi(\cdot)$ then
\bea
\varphi(x(T))-\varphi(x^o(T)) \geq \la \varphi_x(x^o(T)), x(T)-x^o(T)\ra . \label{s2}
\eea
Substituting (\ref{s2}) into (\ref{s1}) yields
 \begin{align}
 J(u^o) -J(u)\leq  {\mathbb E} \Big\{ & \la \varphi_x(x^o(T)),   x^o(T) - x(T)\ra \Big\} \nonumber \\
 + &    {\mathbb E} \biggl\{    \int_{0}^{T}  \Big(\ell(t,x^o(t),u_t^{o})  -\ell(t,x(t),u_t) \Big)  dt   \biggr\}    . \label{s3}
  \end{align}
Applying the Ito differential rule to $\la\psi^o,x-x^o\ra$ on the interval $[0,T]$ and then taking expecation we obtain the following equation.
\begin{align}
{\mathbb E} \Big\{ & \la \psi^o(T),   x(T) - x^o(T)\ra \Big\}
 =  {\mathbb E} \Big\{  \la \psi^o(0),   x(0) - x^o(0)\ra \Big\} \nonumber \\
& +{\mathbb E} \Big\{ \int_{0}^{T}  \la-f_x^{*}(t,x^o(t),u_t^{o})\psi^o(t)dt-V_{Q^o}(t)-\ell_x(t,x^o(t),u_t^{o}), x(t)-x^o(t)\ra dt  \Big\} \nonumber \\
&+ {\mathbb E} \Big\{   \int_{0}^{T}  \la \psi^o(t), f(t,x(t),u_t)- f(t,x^o(t),u_t^{o})\ra dt \Big\} \nonumber \\
& + {\mathbb E} \Big\{  \int_{0}^{T}     tr (Q^{*,o}(t)\sigma(t,x(t),u_t)-Q^{*,o}(t)\sigma(t,x^o(t),u_t^{o}))dt \Big\} \nonumber \\
&=  -  {\mathbb E} \Big\{ \int_{0}^{T} \la {\mathbb H}_x(t,x^o(t),\psi^o(t), Q^o(t),u_t^{o}), x(t)-x^o(t)\ra dt \nonumber \\
& + {\mathbb E} \Big\{  \int_{0}^{T} \la \psi^o(t), f(t,x(t),u_t)-f(t,x^o(t),u_t^{o})\ra dt \Big\} \nonumber \\
& + {\mathbb E} \Big\{  \int_{0}^{T}  tr (Q^{*,o}(t)\sigma(t,x(t),u_t)-Q^{*,o}(t)\sigma(t,x^o(t),u_t^{o}))dt \Big\} \label{s4}
\end{align}
Note that $\psi^o(T)=\varphi_x(x^o(T))$. Substituting (\ref{s4}) into (\ref{s3}) we obtain
 \begin{align}
 J(u^o) -J(u)  \leq &    {\mathbb E} \Big\{ \int_{0}^{T}  \Big[ {\mathbb H}(t,x^o(t),\psi^o(t), Q^o(t),u_t^{o}) -   {\mathbb H}(t,x(t),\psi^o(t), Q^o(t),u_t^{})\Big]dt \Big\} \nonumber \\
 -&    {\mathbb E} \Big\{ \int_{0}^{T} \la {\mathbb H}_x(t,x^o(t),\psi^o(t), Q^o(t),u_t^{o}), x^o(t)-x(t)\ra dt   \Big\}    . \label{s5}
  \end{align}
Since by hypothesis $ {\mathbb H}$ is convex in $\xi \in {\mathbb R}^n$ and linear in $\nu \in {\cal M}_1({\mathbb A}^{(N)})$, ${\mathbb H}$ is convex in both $(\xi, \nu) \in  {\mathbb R}^n \times {\cal M}_1({\mathbb A}^{(N)})$.  Using this fact in (\ref{s5}) we readily obtain

 \begin{align}
  J(u^o) -J(u)  \leq
     {\mathbb E}     \int_{0}^{T}  <{\mathbb H}(t,x^o(t),\psi^o(t), Q^o(t),\cdot), u_t^{o}(\cdot)-u_t(\cdot)> dt   \leq 0, \hst \forall u \in {\mathbb U}_{rel}^{(N)}[0,T],\label{s5a}
  \end{align}
where the last inequality follows from (\ref{eqh35}).    This proves that  $u^o$ optimal and hence  the necessary conditions are also sufficient.

\end{proof}

 Under conditions  similar to those of  Theorem~\ref{theorem5.1s}, we can verify that a strategy is person-by-person optimal  for Problem~\ref{problem2} if it satisfies (\ref{eqh35}); this is stated as a corollary. Indeed, the  necessary conditions for team optimality and person-by-person optimality are equivalent, and person-by-person optimality implies team optimality.

 \begin{theorem} (Sufficient conditions for person-by-person optimality)
 \label{theorem5.1sa}
   Consider Problem~\ref{problem2} and suppose Assumptions~ \ref{assumptionscost}, \ref{NC1} hold.  Under the conditions of Theorem~\ref{theorem5.1},  Part {\bf (I)}, let $(u^o(\cdot), x^o(\cdot))$ denote any  control-state  pair and let $\psi^o(\cdot)$ the corresponding adjoint processes. \\
   Suppose the  conditions  of Theorem~\ref{theorem5.1s}, {\bf (C4), (C5)} hold.\\
  \noi Then $(u^o(\cdot),x^o(\cdot))$ is player-by-player  optimal if it satisfies (\ref{eqh35}). \\
For feedback strategies ${\mathbb U}_{rel}^{(N),z^u}[0,T]$ the above statements hold under the conditions of Theorem~\ref{theorem5.1}, Part {\bf (II)}.
\end{theorem}

\begin{proof} The proof  is  similar to that of  Theorem~\ref{theorem5.1s}.
\end{proof}

\section{Optimality Conditions for Regular Strategies }
\label{regular}
In the development of the necessary and sufficient conditions of optimality given in the previous section we have given conditions  which  assert the existence of optimal decisions from the class of relaxed decisions ${\mathbb U}_{rel}^{(N)}[0,T]$ and ${\mathbb U}_{rel}^{(N),z^u}[0,T]$  in Theorem~\ref{theorem3.2}. \\
The main observation of this section is that, if optimal regular decisions  exist from the admissible class ${\mathbb U}_{reg}^{(N)}[0,T] \subset {\mathbb U}_{rel}^{(N)}[0,T]$ (or the feedback class)  then the necessary and sufficient conditions of Theorem~\ref{theorem5.1} and Theorem~\ref{theorem5.1s} can be specialized to the class of decision strategies which are simply Dirac measures concentrated $\{u_t^o : t \in [0,T]\} \in {\mathbb U}_{reg}^{(N)}[0,T]$ or ${\mathbb U}_{reg}^{(N),z^u}[0,T]$. The important advantage of the theory of relaxed controls is that the necessary conditions of optimality  for ordinary controls follow readily from those of relaxed controls without requiring differentiability of the Hamiltonian or equivalently the drift and the diffusion coefficients  $f, \sigma $ with respect to  the control variables.

\noi Thus we simply state the necessary and sufficient conditions of optimality  for regular decentralized decision strategies  which follow as a corollary of Theorem~\ref{theorem5.1}, \ref{theorem5.1s}  by simply specializing to regular decision strategies given by Dirac measures along the regular decision strategies leading to the following Hamiltonian
\bes
 {\cal  H}: [0, T] \times {\mathbb R}^n\times {\mathbb R}^n\times {\cal L}({\mathbb R}^m,{\mathbb R}^n)\times {\mathbb A}^{(N)} \longrightarrow {\mathbb R},
\ees
   where
   \begin{align}
    {\cal H} (t,\xi,\zeta,M,\nu) \tri \la f(t,\xi,\nu),\zeta\ra + tr (M^*\sigma(t,\xi,\nu))
     + \ell(t,\xi,\nu),  \hst  t \in  [0, T]. \label{dh1}
    \end{align}

\begin{theorem}(Regular team optimality conditions)
\label{theorem7.1}
 Consider Problem~\ref{problem1} under the Assumptions of Theorem~\ref{theorem5.1} with decisions (or controls)  from the regular class taking  values in   ${\mathbb A}^i,$  a closed, bounded and convex subset of ${\mathbb R}^{d_i}$, $\forall  i\in {\mathbb Z}_N$.

\begin{description}

\item[(I)] Let  ${\mathbb F}_T$ denote  the filtration generated by $x(0)$ and the Brownian motion $W$.

\noi {\bf Necessary Conditions.}    For  an element $ u^o \in {\mathbb U}_{reg}^{(N)}[0,T]$ with the corresponding solution $x^o \in B_{{\mathbb F}_T}^{\infty}([0,T], L^2(\Omega,{\mathbb R}^n))$ to be team optimal, it is necessary  that
the following hold.

\begin{description}

\item[(1)]  There exists a semi martingale  $m^o \in {\cal SM}_0^2[0,T]$ with the intensity process $({\psi}^o,Q^o) \in  L_{{\mathbb F}_T}^2([0,T],{\mathbb R}^n)\times L_{{\mathbb F}_T}^2([0,T],{\cal L}({\mathbb R}^m,{\mathbb R}^n))$.

 \item[(2) ]  The variational inequality is satisfied:

\begin{align}     \sum_{i=1}^N {\mathbb  E} \Big\{ \int_0^T   \Big(  {\cal H} (t,x^o(t),\psi^o(t), Q^{o}(t), u_t^i, u_t^{-i,o})&-{\cal H} (t,x^o(t),\psi^o(t), Q^{o}(t), u_t^{o}) \Big) dt \Big\}\geq 0,     \nonumber \\
& \forall u \in {\mathbb U}_{reg}^{(N)}[0,T].  \label{eqd16}    \end{align}

\item[(3)]  The process $({\psi}^o,Q^o) \in  L_{{\mathbb F}_T}^2([0,T],{\mathbb R}^n)\times L_{{\mathbb F}_T}^2([0,T],{\cal L}({\mathbb R}^m,{\mathbb R}^n))$ is a unique solution of the backward stochastic differential equation (\ref{adj1a}), (\ref{eq18}), with  ${\mathbb H}$ replaced by ${\cal H}$ such that $u^o \in {\mathbb U}_{reg}^{(N)}[0,T]$ satisfies  the  point wise almost sure inequalities with respect to the $\sigma$-algebras ${\cal G}_{0,t}^i   \subset {\mathbb F}_{0,t}$, $ t\in [0, T], i=1, 2, \ldots, N:$

\begin{align}
  {\mathbb E} \Big\{ \Big( {\cal H}(t,x^o(t),& \psi^o(t),Q^o(t),u_t^i, u_t^{-i,o})- {\cal H}(t,x^o(t),  \psi^o(t),Q^o(t), u_t^{o}) \Big) |{\cal G}_{0, t}^i \Big\}   \geq  0, \nonumber  \\
&\forall u^i \in {\mathbb A}^i,  a.e. t \in [0,T], {\mathbb P}|_{{\cal G}_{0,t}^i}- a.s., i=1,2,\ldots, N.   \label{eqhd35}   \end{align}

 \noi {\bf Sufficient Conditions.}    Let $(u^o(\cdot), x^o(\cdot))$ denote an admissible decision and state  pair and  $\psi^o(\cdot)$ the corresponding adjoint processes. \\
   Suppose the conditions {\bf (C4), (C5)} holds and in addition

\begin{description}

\item[(C6)]  ${\cal H}(t, \xi,\zeta,M, \cdot),   t \in  [0, T]$,  is convex in $u \in {\mathbb A}^{(N)}$;

\end{description}

\end{description}

\noi Then $(x^o(\cdot),u^o(\cdot))$ is optimal if it satisfies (\ref{eqhd35}).

\item[(II)] Suppose ${\mathbb F}_T$ is the filtration generated by $x(0)$ and the Brownian motion $W$, and Assumptions~\ref{a-nf1} hold with decision policies  from the regular class. The necessary and sufficient conditions for a feedback  policy  $ u^o \in {\mathbb U}_{reg}^{(N),z^u}[0,T]$ to be optimal are given by the statements under Part {\bf (I)} with ${\cal G}_{0,t}^i$ replaced by ${\cal G}_{0,t}^{z^{i,u}}, \forall t \in [0,T]$.

%%%% Start Remove
\begin{comment}
\item[(III)] Suppose  $\{W(t): t \in [0,T]\}$ is an ${\mathbb F}_T-$martingale, and this filtration is right continuous not necessarily generated by $x(0)$ and $\{W(t): t \in [0,T]\}$, and Assumptions~\ref{NCD1} hold with the stronger condition {\bf (D5)}.  The necessary and sufficient conditions for a feedback  element $ u^o \in {\mathbb U}_{reg}^{(N),z^u}[0,T]$ to be optimal are given by the statements under Part {\bf (I)} with ${\cal G}_{0,t}^i$ replaced by ${\cal G}_{0,t}^{z^{i,u}}, \forall t \in [0,T]$, and the adjoint equation (\ref{adjg1}), (\ref{adjg2}) (with ${\mathbb H}$ replaced by ${\cal H}$).
\end{comment}
%%% End Remove

\end{description}

\end{theorem}

\begin{proof} Follows from Theorem~\ref{theorem5.1}, \ref{theorem5.1s} by simply replacing relaxed controls by Dirac measures concentrated at $\{u_t^o : t \in [0,T]\} \in {\mathbb U}_{reg}^{(N)}[0,T]$ or ${\mathbb U}_{reg}^{(N),z^u}[0,T]$.
%
%%%% Start Remove
%\begin{comment}
%
% {\bf (I)}.  This part  follows from straightforward application of the embedding mentioned above and the definition  (\ref{eq49}). Considering  the necessary conditions of optimality given in Theorem ~\ref{theorem5.1}, and using the  embedding mentioned above it is easy to derive the statements in the necessary conditions of optimality. The sufficiency part follows from the above embedding and the additional convexity conditions as in the case of relaxed decisions. \\
%{\bf (II).}  By Theorem~\ref{thm-nf1r},  $\inf_{ u \in   {\mathbb U}_{reg}^{(N),z^u}[0,T] } J(u) = \inf_{ u \in   \overline{\mathbb U}_{reg}^{(N)}[0,T] } J(u)$.  Since, this holds then the necessary conditions can be derived using strategies from  $\overline{\mathbb U}_{reg}^{(N)}[0,T]$. The sufficiency part is done as in {\bf (I)}.  \\
%{\bf (III)} This is similar to Theorem~\ref{theorem5.1}, Part {\bf (III)}.
%
%\end{comment}

%%% End Remove

\end{proof}

 Person-by-person optimality conditions for regular decision strategies follow from their  relaxed  counterparts, as discussed above. Therefore we simply state the results as a corollary.

\begin{corollary}(Person-by-person optimality)
\label{corollarypbpd}
Consider Problem~\ref{problem2} under the conditions of Theorem~\ref{theorem7.1}. Then the necessary and sufficient conditions of Theorem~\ref{theorem7.1} hold with the variational inequality (\ref{eqd16}) replaced by
\begin{align}     {\mathbb  E} \Big\{ \int_0^T  \Big(& {\cal H}(t,x^o(t),\psi^o(t), Q^{o}(t),u_t^i, u_t^{-i,o} ) \nonumber \\
&- {\cal H}(t,x^o(t),\psi^o(t), Q^{o}(t),u_t^{i,o}, u_t^{-i,o}) \Big) dt \Big\}\geq 0, \hst \forall u^i \in {\mathbb U}_{reg}^{i}[0,T], \hso \forall i \in {\mathbb Z}_N. \label{eqd16pp}
\end{align}
 Similar conclusions hold for strategies ${\mathbb  U}_{reg}^{z^{i,u}}[0,T]$.
\end{corollary}

\begin{proof}  Follows from Corollary~\ref{corollarypbpd} by simply replacing relaxed controls by Dirac measures concentrated at $\{u_t^o : t \in [0,T]\} \in {\mathbb U}_{reg}^{(N)}[0,T]$ or ${\mathbb U}_{reg}^{(N),z^u}[0,T]$.

%%% Begin Remove
\begin{comment}
Let us describe the key steps, since the procedure is based on the derivation of Theorem~\ref{theorem7.1}, which utilizes the derivation of Theorem~\ref{theorem5.1}. Note that the variational equation of $Z^i$ satisfies  (\ref{eq9dt})  with the last two right hand side terms replaced by $f_{u^i}(t,x^o(t),u_t^o)(u_t^i-u_t^{i,o})dt$ and  $\sigma_{u^i}(t,x^o(t),u_t^o;u_t^i-u_t^{i,o})dW(t)$ and the Gateaux derivative is given by (\ref{eq20dt}) with the summation been over the single direction $(u_t^i-u_t^{i,o})$. Hence, the derivations is done similar to Theorem~\ref{theorem7.1} (or Theorem~\ref{theorem5.1}).
\end{comment}

%%% End Remove

\end{proof}

The optimality conditions are derived based on the assumption that  the filtration ${\mathbb F}_T$ is generated by the system Brownian motions $\{W(t): t \in [0,T]\}$. When this condition does not hold the optimality conditions are  slightly  modified as discussed in the next remark.

\begin{remark}
\label{martgeneral}
Suppose ${\mathbb F}_T$ is not generated by Brownian motions $\{W(t): t \in [0,T]\}$ but stochastic integrals with respect to $W(\cdot)$ are ${\mathbb F}_T-$martingales. Then by invoking the variation of the semi martingale representation due to Kunita-Watanabe (for the derivation see \cite{dellacherie-meyer1980}) we have the following. If  $(i): L^2(\Omega, {\mathbb F}, {\mathbb P})$ is separable and  (ii):  ${\mathbb F}_T$ is right continuous having left limits,  then any square integrable ${\mathbb F}_T$ martingale has the decomposition
\begin{eqnarray}
 m(t) = m(0) + \int_0^t v(s) ds + \int_0^t \Sigma(s) dW(s) +  M(t), \hst t \in [0,T], \label{eq12a}
  \end{eqnarray}
   for some $v \in L_{{\mathbb F}_T}^2([0,T],{\mathbb R}^n)$, $\Sigma \in L_{{\mathbb F}_T}^2([0,T],{\cal L}({\mathbb R}^m,{\mathbb R}^n))$,  ${\mathbb R}^n-$valued ${\mathbb  F}_{0,0}-$measurable  random variable  $m(0)$ having finite second moment, and  $\{M(t): t \in [0,T]\}$ right continuous square integrable ${\mathbb F}_T$ martingale, which is orthogonal to $\{W(t): t \in [0,T]\}$. This  representation is unique. Further, the  stochastic integrals    $\int_0^t \Sigma(s) dW(s)$ and $\int_{0}^t \Gamma(s)dM(s)$ are  orthogonal martingales for $L_2$ integrands.
   In this case  the adjoint equation given by   (\ref{adj1a}), (\ref{eq18})  is replaced by
\begin{align}
d\psi (t)  =&- {\mathbb H}_x (t,x(t),\psi(t),Q(t),u_t) dt + Q(t)dW(t) + dM(t),  \hst t \in [0,T) \label{adjg1} \\
 \psi(T)=&   \varphi_x(x(T)).  \label{adjg2}
 \end{align}

   \end{remark}

 In view of the results obtained, we confirm that  there are no limitations in applying classical theory of optimization to decentralized systems.  Rather,  the challenge is  in the  implementation of  the new variational Hamiltonians and the computation the optimal strategies for specific examples. In Part II \cite{charalambous-ahmedFIS_Partii2012} of this two-part paper, we shall apply these  optimality conditions to investigate various linear and nonlinear distributed stochastic team games and obtain closed form expressions for the optimal strategies for some of them.

\section{Conclusions and Future Work}
\label{cf}
In this paper we have   considered team games for distributed stochastic dynamical decision systems, with decentralized noiseless information patterns for each DM, under relaxed and deterministic strategies.   Necessary and sufficient optimality conditions with respect to  team optimality and person-by-person optimality criteria are derived, based on Stochastic Pontryagin's minimum principle, while we also discussed existence of the optimal strategies.\\
The methodology is very general, and applicable to many areas. However, several additional issues remain to be investigated. Below, we provide a short list.

\begin{description}

\item[(F1)] For team games with regular strategies and  non-convex action spaces ${\mathbb A}^i, i =1,2,\ldots, N$, if the diffusion coefficients depend on the decision variables then it is necessary to derive optimality conditions based on second-order variations. The methodology presented to derive the necessary conditions of optimality can be easily extended to cover this case as well.

%%% Begin Remove
\begin{comment}
\textbf{Nasir: The first statement is wrong. We are dealing with relaxed controls and nothing can be gained by needle variation. }

\end{comment}
%%% End Remove

\item[(F2)] The derivation of optimality conditions can be used in other type of games such as Nash-equilibrium games with decentralized information structures for each DM, and minimax games.

\item[(F3)]   The  optimality conditions can be extended to distributed stochastic dynamical  decision systems driven by both continuous Brownian motion processes and jump processes, such as L\'evy or Poisson jump processes, by following the procedure of centralized strategies in \cite{ahmed-charalambous2012a}.

\item[(F4)] The optimality conditions can be applied to specific examples with  decentralized  noiseless information structures. Some of these are presented in the companion paper \cite{charalambous-ahmedFIS_Partii2012}.

\item[(F5)] The methodology can be extended to cover decentralized partial (noisy) information structures.

\end{description}

\bibliographystyle{IEEEtran}
\bibliography{bibdata}

\begin{thebibliography}{10}
\providecommand{\url}[1]{#1}
\csname url@samestyle\endcsname
\providecommand{\newblock}{\relax}
\providecommand{\bibinfo}[2]{#2}
\providecommand{\BIBentrySTDinterwordspacing}{\spaceskip=0pt\relax}
\providecommand{\BIBentryALTinterwordstretchfactor}{4}
\providecommand{\BIBentryALTinterwordspacing}{\spaceskip=\fontdimen2\font plus
\BIBentryALTinterwordstretchfactor\fontdimen3\font minus
  \fontdimen4\font\relax}
\providecommand{\BIBforeignlanguage}[2]{{%
\expandafter\ifx\csname l@#1\endcsname\relax
\typeout{** WARNING: IEEEtran.bst: No hyphenation pattern has been}%
\typeout{** loaded for the language `#1'. Using the pattern for}%
\typeout{** the default language instead.}%
\else
\language=\csname l@#1\endcsname
\fi
#2}}
\providecommand{\BIBdecl}{\relax}
\BIBdecl

\bibitem{fleming-rischel1975}
W.~Fleming and R.~Rischel, \emph{Deterministic and Stochastic Optimal
  Control}.\hskip 1em plus 0.5em minus 0.4em\relax Springer Verlag, 1975.

\bibitem{elliott1977}
R.~J. Elliott, ``The optimal control of stochastic system,'' \emph{SIAM Journal
  on Control and Optimization}, vol.~15, no.~5, pp. 756--778, 1977.

\bibitem{bismut1978}
J.~M. Bismut, ``An introductory approach to duality in optimal stochastic
  control,'' \emph{SIAM Review}, vol.~30, pp. 62--78, 1978.

\bibitem{ahmed1981}
N.~U. Ahmed and K.~L. Teo, \emph{Optimal Control of Distributed Parameter
  Systems}.\hskip 1em plus 0.5em minus 0.4em\relax Elsevier North Holland, New
  York, Oxford, 1981.

\bibitem{elliott1982}
R.~J. Elliott, \emph{Stochastic Calculus and Applications}.\hskip 1em plus
  0.5em minus 0.4em\relax Springer-Verlag, 1982.

\bibitem{elliott-kohlmann1994}
R.~J. Elliott and M.~Kohlmann, ``The second order minimum principle and adjoint
  process,'' \emph{Stochastics \& Stochastic Reports}, vol.~46, pp. 25--39,
  1994.

\bibitem{peng1990}
S.~Peng, ``A general stochastic maximum principle for optimal control
  problems,,'' \emph{SIAM Journal on Control and Optimization}, vol.~28, no.~4,
  pp. 966--979, 1990.

\bibitem{yong-zhou1999}
J.~Yong and X.~Y. Zhou, \emph{Stochastic Controls, Hamiltonian Systems and
  {HJB} Equations}.\hskip 1em plus 0.5em minus 0.4em\relax Springer-Verlag,
  1999.

\bibitem{ahmed2006}
N.~U. Ahmed, \emph{Dynamic Systems and Control with Applications}.\hskip 1em
  plus 0.5em minus 0.4em\relax World Scientific, New Jersey London, Singapore
  Beijing Shanghai, Hong Kong, Taipei, Chenna, 2006.

\bibitem{bensoussan1983}
A.~Bensoussan, ``Maximum principle and dynamic programming approaches of the
  optimal control of partially observed diffusions,'' \emph{Stochastics},
  vol.~9, no. 169-222, 1983.

\bibitem{bensoussan1992a}
------, \emph{Stochastic Control of Partially Observable Systems}.\hskip 1em
  plus 0.5em minus 0.4em\relax Cambridge University Press, 1982.

\bibitem{charalambous-hibey1996}
C.~D. Charalambous and J.~L. Hibey, ``Minimum principle for partially
  observable nonlinear risk-sensitive control problems using measure-valued
  decompositions,'' \emph{Stochastics \& Stochastic Reports}, pp. 247--288,
  1996.

\bibitem{ahmed-charalambous2007}
N.~U. Ahmed and C.~D. Charalambous, ``Minimax games for stochastic systems
  subject to relative entropy uncertainty: Applications to {SDE}'s on {H}ilbert
  spaces,'' \emph{Journal of Mathematics of Control, Signals and System},
  vol.~19, pp. 197--216, 2007.

\bibitem{ahmed2005}
N.~U. Ahmed, ``Optimal relaxed controls for systems governed by impulsive
  differential inclusions,'' \emph{Nonlinear Functional Analysis \&
  Applications}, vol.~10, no.~3, pp. 427--460, 2005.

\bibitem{witsenhausen1968}
H.~S. Witsenhausen, ``A counter example in stochastic optimum control,''
  \emph{SIAM Journal on Control and Optimization}, vol.~6, no.~1, pp. 131--147,
  1968.

\bibitem{witsenhausen1971}
------, ``Separation of estimation and control for discrete time systems,'' in
  \emph{Proceedings of the IEEE}, 1971, pp. 1557--1566.

\bibitem{ho-chu1972}
Y.-C. Ho and K.-C. Chu, ``Team decision theory and information structures in
  optimal control problems-part i,'' \emph{IEEE Transactions on Automatic
  Control}, vol.~17, no.~1, pp. 15--22, 1972.

\bibitem{kurtaran-sivan1973}
B.-Z. Kurtaran and R.~Sivan, ``Linear-{Q}uadratic-{G}aussian control with
  one-step-delay sharing pattern,'' \emph{IEEE Transactions on Automatic
  Control}, pp. 571--574, 1974.

\bibitem{sandell-athans1974}
N.~R. Sandell and M.~Athans, ``Solution of some nonclassical {LQG} stochastic
  decision problems,'' \emph{IEEE Transactions on Automatic Control}, vol.~19,
  no.~2, pp. 108--116, 1974.

\bibitem{kurtaran1975}
B.-Z. Kurtaran, ``A concice derivation of the {LQG} one-step-delay sharing
  problem solution,'' \emph{IEEE Transactions on Automatic Control}, vol.~20,
  no.~6, pp. 808--810, 1975.

\bibitem{varaiya-walrand1978}
P.~Varaiya and J.~Walrand, ``On delay sharing patterns,'' \emph{IEEE
  Transactions on Automatic Control}, vol.~23, no.~3, pp. 443--445, 1978.

\bibitem{ho1980}
Y.~Ho, ``Team decision theory and information structures,'' \emph{Proceedings
  of IEEE}, vol.~68, pp. 644--655, 1980.

\bibitem{krainak-speyer-marcus1982a}
J.~Krainak, J.~L. Speyer, and S.~I. Marcus, ``Static team problems-part {I}:
  Sufficient conditions and the exponential cost criterion,'' \emph{IEEE
  Transactions on Automatic Control}, vol.~27, no.~4, pp. 839--848, 1982.

\bibitem{krainak-speyer-marcus1982b}
------, ``Static team problems-part {II}: Affine control laws, projections,
  algorithms, and the {LEGT} problem,'' \emph{IEEE Transactions on Automatic
  Control}, vol.~27, no.~4, pp. 848--859, 1982.

\bibitem{bansal-basar1987}
R.~Bansar and T.~Basar, ``Stochastic teams with nonclassical information
  revisited: When is an affine law optimal,'' \emph{IEEE Transactions on
  Automatic Control}, vol.~32, no.~6, pp. 554--559, 1987.

\bibitem{waal-vanschuppen2000}
P.~R. Wall and J.~H. van Schuppen, ``A class of team problems with discrete
  action spaces: Optimality conditions based on multimodularity,'' \emph{SIAM
  Journal on Control and Optimization}, vol.~38, no.~3, pp. 875--892, 2000.

\bibitem{bamieh-voulgaris2005}
B.~Bamieh and P.~Voulgaris, ``A convex characterization of distributed control
  problems in spatially invariant systems with communication constraints,''
  \emph{Systems and Control Letters}, vol.~54, no.~6, pp. 575--583, 2005.

\bibitem{aicardi-davoli-minciardi1987}
M.~Aicardi, F.~Davoli, and R.~Minciardi, ``Decentralized optimal control of
  markov chains with a common past information,'' \emph{IEEE Transactions on
  Automatic Control}, vol.~32, no.~11, pp. 1028--1031, 1987.

\bibitem{nayyar-mahajan-teneketzis2011}
A.~Nayyar, A.~Mahajan, and D.~Teneketzis, ``Optimal control strategies in
  delayed sharing information structures,'' \emph{IEEE Transactions on
  Automatic Control}, vol.~56, no.~7, pp. 1606--1620, 2011.

\bibitem{vanschuppen2011}
J.~H. van Schuppen, ``Control of distributed stochastic systems-introduction,
  problems, and approaches,'' in \emph{International Proceedings of the IFAC
  World Congress}, 2011.

\bibitem{lessard-lall2011}
L.~Lessard and S.~Lall, ``A state-space solution to the two-player optimal
  control problems,'' in \emph{Proceedings of 49th Annual Allerton Conference
  on Communication, Control and Computing}, 2011.

\bibitem{mahajan-martins-rotkowitz-yuksel2012}
A.~Mahajan, N.~Martins, M.~Rotkowitz, and S.~Yuksel, ``Information structures
  in optimal decentralized control,'' in \emph{In Proceedings of the 51st
  Conference on Decision and Control}, 2011.

\bibitem{marschak1955}
J.~Marschak, ``Elements for a theory of teams,'' \emph{Management Science},
  vol.~1, no.~2, 1955.

\bibitem{radner1962}
R.~Radner, ``Team decision problems,'' \emph{The Annals of Mathematical
  Statistics}, vol.~33, no.~3, pp. 857--881, 1962.

\bibitem{marschak-radner1972}
J.~Marschak and R.~Radner, \emph{Economic Theory of Teams}.\hskip 1em plus
  0.5em minus 0.4em\relax New Haven: Yale University Pres, 1972.

\bibitem{ahmed-charalambous2012a}
N.~U. Ahmed and C.~D. Charalambous, ``Stochastic minimum principle for
  partially observed systems subject to continuous and jump diffusion processes
  and driven by relaxed controls,'' \emph{SIAM Journal on Control and
  Optimization}, 2012, submitted, June 2012.

\bibitem{vanschuppen2012}
J.~H. van Schuppen, O.~Boutin, P.~L. Kempker, J.~Komenda, T.~Masopust,
  N.~Pambakian, and A.~C.~M. Ran, ``Control of distributed systems: Tutorial
  and overview,'' \emph{European Journal on Control}, pp. 579--602, 2012.

\bibitem{charalambous-ahmedFIS_Partii2012}
C.~D. Charalambous and N.~U. Ahmed, ``Centralized versus decentralized team
  games of distributed stochastic differential decision systems with noiseless
  information structures-{P}art {II}: Applications,'' \emph{Preprint}, 2012,
  draft: October 2012.

\bibitem{liptser-shiryayev1977}
R.~Liptser and A.~Shiryayev, \emph{Statistics of Random Processes Vol.1}.\hskip
  1em plus 0.5em minus 0.4em\relax Springer-Verlag New York, 1977.

\bibitem{bensoussan1981}
A.~Bensoussan, \emph{Lecture on Stochastic Control, Lecture Notes in
  Mathematics}.\hskip 1em plus 0.5em minus 0.4em\relax Springer-Verlag, Berlin,
  1982.

\bibitem{dellacherie-meyer1980}
C.~Dellacherie and P.~Meyer, \emph{Probabilites et Potentiel}.\hskip 1em plus
  0.5em minus 0.4em\relax Hermann, Paris, •, ch. Chapitres {I} and {IV}.

\end{thebibliography}

\end{document}